\documentclass{amsart}

\usepackage{amssymb,amsfonts}
\usepackage[all,arc]{xy}
\usepackage{enumitem}
\usepackage{mathrsfs}
\usepackage{verbatim}
\usepackage{microtype}
\usepackage{mathtools}

\newcounter{derpp}

\newtheorem{thm}{Theorem}[section]
\newtheorem*{thm*}{Theorem}
\newtheorem{cor}[thm]{Corollary}
\newtheorem*{cor*}{Corollary}
\newtheorem{prop}[thm]{Proposition}
\newtheorem{lem}[thm]{Lemma}
\newtheorem{conj}[thm]{Conjecture}
\newtheorem*{conj*}{Conjecture}
\newtheorem{quest}[thm]{Question}

\theoremstyle{definition}
\newtheorem{defn}[thm]{Definition}

\newtheorem{exmp}[thm]{Example}
\newtheorem{exmps}[thm]{Examples}

\theoremstyle{remark}
\newtheorem{rem}[thm]{Remark}

\DeclareMathOperator{\Spec}{Spec}

\DeclareMathOperator{\codim}{codim}
\DeclareMathOperator{\Tor}{Tor}

\DeclareMathOperator{\Supp}{Supp}
\DeclareMathOperator{\ord}{ord}

\DeclareMathOperator{\trdeg}{tr.deg}

\DeclareMathOperator{\Zar}{Zar}
\DeclareMathOperator{\ann}{ann}
\DeclareMathOperator{\hgt}{ht}
\DeclareMathOperator{\ch}{CH}
\DeclareMathOperator{\Nis}{Nis}

\newcommand{\mcf}{\mathcal{F}}
\newcommand{\mcg}{\mathcal{G}}
\newcommand{\mco}{\mathcal{O}}
\newcommand{\mce}{\mathcal{E}}
\newcommand{\mck}{\mathcal{K}}
\newcommand{\mcm}{\mathcal{M}}
\newcommand{\fm}{\mathfrak{m}}
\newcommand{\fn}{\mathfrak{n}}
\newcommand{\fp}{\mathfrak{p}}
\newcommand{\fq}{\mathfrak{q}}

\newcommand{\ds}{\displaystyle}

\newcommand{\ovp}{\frac{1}{p}}
\newcommand{\ovpi}{\frac{1}{\pi}}
\newcommand{\fcolim}{\ds \varinjlim_{\lambda \in \Lambda}}
\newcommand{\finvlim}{\ds \varprojlim_{\lambda \in \Lambda}}
\newcommand{\mbf}{\mathbf{F}}
\newcommand{\modl}{\mathbb{Z}/\ell \mathbb{Z}}
\newcommand{\dumcg}{\underline{\underline{\mathcal{G}}}}

\makeatletter
\let\c@equation\c@thm
\makeatother
\numberwithin{equation}{section}

\address{Department of Mathematics \\
University of Illinois at Chicago \\
851 S. Morgan St. \\
Chicago, IL 60607}
\email[Chris Skalit]{cskali2@uic.edu}
\title{Regular Morphisms and Gersten's Conjecture}
\author{C. Skalit}

\begin{document}
\begin{abstract} We prove that when $X \to Y$ is a (geometrically) regular morphism of Noetherian schemes, then from a Nisnevich-local perspective, the Gersten complex for Quillen $K$-theory on $X$ becomes acyclic in degrees beyond the Krull dimension of $Y$. Using our methods, we also reduce the general Gersten conjecture for regular, unramified local rings to the case of a discrete valuation ring which is essentially smooth over $\mathbb{Z}$. We apply our results to the the theory of algebraic cycles --- globally to obtain relative versions of Bloch's Formula and locally to address the Claborn-Fossum Conjecture concerning the vanishing of Chow groups for regular local rings.  \end{abstract}
\maketitle

\section{Introduction}Fix a Noetherian scheme $X$. From the classical work of Quillen \cite[VII.5]{Quillen}, the codimensional filtration on the category of coherent $\mco_X$-modules gives, for each $n \geq 0$, a ``Gersten complex'' $\mcg_n(X)$:
\[ 0 \to \bigoplus_{x \in X^0}K_n(k(x)) \to \bigoplus_{x \in X^1}K_{n-1} \to \cdots \to \bigoplus_{x \in X^{n-1}}K_1(k(x)) \to \bigoplus_{x \in X^n}K_0(k(x)) \to 0 \]
where $X^p$ denotes the set of all points having codimension $p$ in $X$. The present article concerns the following acyclicity conjecture:

\begin{conj*}[Gersten's Conjecture] \cite[VII.5.10]{Quillen} Let $X = \Spec A$ where $(A,\fm_A,L)$ is a regular local ring. Then $\mcg_n(X)$ resolves $K_n(X)$ in the following sense:
\[ H^p(\mcg_n(X)) = \left\{ \begin{array}{ll} K_n(X) & p = 0 \\ 0 & p > 0 \end{array} \right. . \]
\end{conj*}

Quillen [ibid.] proved this conjecture when $A$ is the local ring of a smooth variety defined over a field. The generalization to equicharacteristic regular local rings came much later from the work of Panin \cite{Panin}. To date, the case of mixed characteristic remains largely open; the most general result in this direction, due to Gillet and Levine \cite{GilletLevine}, states that when $A$ is local and essentially smooth over a DVR, then $\mcg_n^{\cdot}(A)$ is acyclic in degree $2$ and higher. Our main theorem aims to generalize this result in two directions: first by replacing the base DVR with an arbitrary local ring and second by replacing ``essentially smooth'' with ``geometrically regular\footnote{Many authors use the term ``regular morphism'' for what we shall call a ``geometrically regular morphism.'' See Definition \ref{g_reg} and the remark that follows it.},'' thereby eliminating all finite-type hypotheses. We do this at the cost of Henselization, thus producing a Nisnevich-local result:

\begin{thm*}Let $(R,\fm_R,k) \to (A,\fm_A,L)$ be a geometrically regular morphism of Noetherian local rings, and assume $A$ is Henselian. Then $H^p(\mcg_n^{\cdot}(A)) = 0$ for all $p > \dim R$ and all $n \geq 0$.
\end{thm*}

The theorem, which appears below as Theorem \ref{mainthm}, is proved in two main steps. First, in Section \ref{finite_case}, we handle the theorem in the finitistic case for $A_{\lambda}$ the Henselization of an essentially smooth $R$-algebra. This requires a number of normalization lemmas, many of which involve the notion of formal smoothness, which we recall in Section \ref{formal_smoothness}. 

For the second step, we use Popescu's Theorem \cite{Popescu} to realize $A$ as a filtered colimit of local rings $A_{\lambda}$, each the Henselization of an essentially smooth $R$-algebra. We cannot blindly pass to the limit since the Gersten complex is only functorial with respect to flat maps, and the transition maps $A_{\lambda} \to A_{\lambda'}$ in our filtered system are far from flat. To bypass this obstacle, we introduce a filtration on the Gersten complex of $A$ in Section \ref{gersten_filter}, giving rise to a spectral sequence whose $E_1$ page contains the data of the Gersten complex on each fibre ring of the map $R \to A$. Since fibre rings are equicharacteristic and regular, we use Panin's result \cite{Panin} to recast Gersten cohomology as Zariski cohomology of the $K$-theory sheaves $\mck_n$ (see Section \ref{popescu_consequence}). As the fibre rings of $R \to A$ are approximated by those of $R \to A_{\lambda}$, we can invoke Grothendieck's limit theorem to kill off enough of the $E_1$ page to prove our theorem.

In Section \ref{unramified}, we turn our attention to unramified regular local rings of mixed characteristic. As they are necessarily geometrically regular over $\mathbb{Z}_{(p)}$, we use Popescu's Theorem to sharpen the results of \cite{GilletLevine} and make the following reduction (Theorem \ref{unramifiedconj} below):

\begin{thm*}The following conjectures are equivalent:
\begin{enumerate}
\item[(a)] Gersten's Conjecture for a discrete valuation ring $(R,pR,k)$, essentially smooth over $\mathbb{Z}$.
 \item[(b)] Gersten's Conjecture for an unramified, regular local ring of mixed characteristic $(p,0)$.
\end{enumerate}
\end{thm*}

Any mixed characteristic DVR which is essentially of finite-type over $\mathbb{Z}$ has a residue field that is finitely-generated over $\mathbb{F}_p$. This, in turn, opens the door to a host of arithmetic techniques. We close Section \ref{unramified} by establishing the relationship with the Beilinson-Parshin Conjecture \cite{Kahn} and then proving the Gersten Conjecture for DVRs with sufficiently ``small'' residue field.

The most immediate corollary of Quillen's original proof of Gersten's Conjecture in the geometric case is Bloch's Formula \cite[VII.5.19]{Quillen}, which states that for a smooth variety $X$, there are isomorphisms
\[ H^p_{\Zar}(X,\mck_n) \cong \ch^p(X). \]
In Section \ref{chow}, we establish analogues of Bloch's Formula for our relative setting by relating the Chow groups of a Noetherian scheme $X$ to various Zariski or Nisnevich hypercohomology groups.

When $(A,\fm_A,L)$ is a regular local ring, the classical Claborn-Fossum Conjecture \cite{ClabornFossum} states that $\ch^p(A) = 0$ for $p > 0$. From our preceding results on the Gersten complex, we obtain the following as Corollary \ref{cf_results} below:

\begin{cor*}Let $(R,\fm_R,k) \to (A,\fm_A,L)$ be a geometrically regular map of Noetherian local rings.
\begin{enumerate}
\item[(i)] If $A$ is Henselian, $\ch^n(A) = 0$ for all $n > \dim R$.
\item[(ii)] If $R$ is a discrete valuation ring, then $\ch^n(A) = 0$ for all $n > 0$.
\end{enumerate}
\end{cor*}

\section{Background Material}
In the interest of self-containment, we review here some notions about the coniveau filtration, the Nisnevich topology, and the relationship between formal smoothness and geometric regularity. It is recommended that the reader skip this section and refer back to it only as needed.

\subsection{General $K$-Theory Conventions}For a Noetherian scheme $X$, we denote by $\mcm(X)$ the category of coherent $\mco_X$-modules and $\mathbf{Vect}(X) \subset \mcm(X)$ the full subcategory of locally-free sheaves. We shall write $K_n'(X) = K_n(\mcm(X))$ and $K_n(X) = K_n(\mathbf{Vect}(X))$. When $X$ is regular and separated over $\Spec \mathbb{Z}$, every $\mcf \in \mcm(X)$ has a finite resolution by locally-free sheaves; so by Quillen's resolution theorem \cite[IV.3]{Quillen}, the natural map $K_n(X) \to K'_n(X)$ is an isomorphism. We shall use $K_n(X)$ and $K'_n(X)$ interchangeably in this context.

\subsection{The Codimensional Filtration and the Gersten Complex}\label{coniveau} We briefly recall the construction of the Gersten Complex. Standard references for this material are \cite[VII.5]{Quillen}, \cite[V.9]{Kbook}, and \cite[Appx. C]{Srinivas}.

Fix a Noetherian scheme $X$ of finite Krull dimension. The category of coherent $\mco_{X}$-modules $\mcm(X)$ comes equipped with a natural filtration
\[ \mcm(X) = \mcm^0(X) \supset \mcm^1(X) \supset \mcm^2(X) \supset \cdots \supset \mcm^p(X) \supset \mcm^{p+1}(X) \supset \cdots \]
where $\mcm^p(X)$ is the full subcategory of $\mcm(X)$, consisting of modules supported in codimension \emph{at least} $p$. In other words, $\mcf \in \mcm^p(X)$ if and only if $\mcf_{x} = 0$ for all $x \in X$ such that $\dim \mco_{X,x} < p$. From this characterization, it readily follows that $\mcm^{p+1}(X)$ is a Serre subcategory of $\mcm^p(X)$. Thus, the quotient $\mcm^p(X)/\mcm^{p+1}(X)$ is abelian. Furthermore, there is a natural equivalence
\[ \mcm^p(X)/\mcm^{p+1}(X) \stackrel{\sim}{\longrightarrow} \coprod_{x \in X^p}\mcm^{\operatorname{fl}}(\mco_{X,x}) \]
where $X^p := \left\{x \in X: \dim \mco_{X,x} = p \right\}$ and $\mcm^{\operatorname{fl}}(\mco_{X,x})$ is the category of finite-length $\mco_{X,x}$ modules. By devissage, one then obtains
\[ K_n(\mcm^p(X)/\mcm^{p+1}(X)) \cong \bigoplus_{x \in X^p}K_n(k(x)) \mbox{ for all $p,n \geq 0$.} \]
Using Quillen's localization theorem, the canonical maps of $\mcm^{p+1}(X) \stackrel{\mathbf{i}}{\longrightarrow} \mcm^p(X) \stackrel{\mathbf{q}}{\longrightarrow} \mcm^p(X)/\mcm^{p+1}(X)$ give rise to long exact sequences
\begin{multline*}
 \cdots \stackrel{\partial}{\longrightarrow} K_n(\mcm^{p+1}(X)) \stackrel{\mathbf{i}_*}{\longrightarrow} K_n(\mcm^p(X)) \stackrel{\mathbf{q}_*}{\longrightarrow} \\
K_n(\mcm^p(X)/\mcm^{p+1}(X)) \stackrel{\partial}{\longrightarrow} K_{n-1}(\mcm^{p+1}(X)) \stackrel{\mathbf{i}_*}{\longrightarrow} \cdots
\end{multline*}
We can patch these localization sequences together, forming the exact couple $(D_1,E_1)$ where $\ds D_1^{pq} = K_{-p-q}(\mcm^p(R))$ and $E_1^{pq} = K_{-p-q}(\mcm^p(X)/\mcm^{p+1}(X))$. If $X$ has finite Krull-dimension, the filtration on $\mcm(X)$ is finite and the associated spectral sequence converges:

\begin{thm}[Coniveau Spectral Sequence]\label{BGQ}Let $X$ be a finite-dimensional Noetherian scheme. There is then a convergent fourth-quadrant spectral sequence
\[ E_1^{pq} = \bigoplus_{x \in X^p}K_{-p-q}(k(x)) \Rightarrow K'_{-p-q}(X). \]
\end{thm}

The $d_1$ differential makes each row of the $E_1$ page a cohomologically-graded complex, which we shall call the Gersten complex on $X$:
\begin{multline*} 
\mcg_n^{\cdot}(X) := \left(0 \to \bigoplus_{\mathclap{x \in X^0}}K_n(k(x)) \to \bigoplus_{\mathclap{x \in X^1}}K_{n-1}(k(x)) \right. \to \cdots \\ 
\left. \cdots \to \bigoplus_{x \in X^p}K_{n-p}(k(x)) \to \cdots \to \bigoplus_{\mathclap{x \in X^{n-1}}}K_1(k(x)) \to \bigoplus_{\mathclap{x \in X^n}}K_0(k(x)) \to 0 \right).
\end{multline*}

To be clear, our indexing starts at degree zero: $\ds \mcg_n^{p}(X) = \bigoplus_{\mathclap{x \in X^p}}K_{n-p}(k(x))$. We also remark that when $X = \Spec A$ is affine, we will write $\mcg_n^{\cdot}(A)$ for $\mcg_n^{\cdot}(X)$ and vice versa.

It is well-known that if $f:X \to Y$ is a flat morphism of Noetherian schemes, then $f^{*}(\mcm^p(Y)) \subset \mcm^p(X)$ \cite[VII.5.2]{Quillen}, thereby inducing a morphism from the coniveau spectral sequence on $Y$ to that on $X$. Proper pushfoward fails to respect the filtration (take $p: \mathbb{P}^2_{\mathbb{C}} \to \Spec \mathbb{C}$ and consider the structure sheaf of a hyperplane). Finite pushforwards, on the other hand, are well-behaved:

\begin{lem}\label{finite_coniveau}Let $g:X \to Y$ be a finite morphism of Noetherian schemes. Then for all $p \geq 0$, $g_{*}(\mcm^p(X)) \subset \mcm^p(Y)$.
\end{lem}
\begin{proof}The question is clearly local on the base, and since $g$ is affine, we reduce to the case where $X = \Spec B$ and $Y = \Spec A$. Let $M \in \mcm^p(B)$ be a finitely-generated $B$-module. Fix $\fp \in \Spec A$ such that when $M$ is regarded as an $A$-module $M_{\fp} \neq 0$. We need to show that $\dim A_{\fp} \geq p$. Note that $A_{\fp} \to B \otimes_A A_{\fp}$ is finite, so $\dim(B \otimes_A A_{\fp}) \leq \dim A_{\fp}$. On the other hand, $0 \neq M_{\fp} = M \otimes_A A_{\fp} = M \otimes_B (B \otimes_A A_{\fp})$, so for some $\fq \in \Spec (B \otimes_A A_{\fp}) \subset \Spec B$, $M_{\fq} \neq 0$. Thus,
\[ \dim A_{\fp} \geq \dim (B \otimes_A A_{\fp}) \geq \dim B_{\fq} \geq p. \]
\end{proof}

\subsection{The Nisnevich Topology}\label{nisnevich} For a Noetherian scheme $Y$, we say that a collection of \'{e}tale morphisms $p_i: U_i \to Y$ is a Nisnevich covering of $Y$ if for every $y \in Y$, the canonical inclusion $j_y: \left\{y\right\}= \Spec k(y) \to Y$ factors through some $p_i$. For $X$ Noetherian, we define the Nisnevich site $X_{\Nis}$ as the collection of finite-type \'{e}tale maps $Y \to X$ with their Nisnevich coverings. See \cite[\S 12]{Mazza} or \cite{Nisnevich} for further details.

For each $n \geq 0$, $K_n(-)$ may be regarded as a presheaf on $X_{\Nis}$ where the restriction maps are given by the flat pullback. We shall denote by $\mck^{\Nis}_{n,X}$, $\mck^{\Zar}_{n,X}$ the respective Nisnevich and Zariski sheafifications, but we shall sometimes suppress the subscripts and superscripts when the context is clear.

For $X$ Noetherian, we can define the Gersten complex of sheaves on $X$ via
\begin{multline*}
\dumcg_{n,X}^{\cdot} = \left(0 \to \bigoplus_{x \in X^0}(j_x)_*\left(\mck^{\Nis}_{n,\left\{x\right\}}\right) \to \bigoplus_{x \in X^1}(j_x)_*\left(\mck^{\Nis}_{n-1,\left\{x\right\}}\right) \to \cdots \right.\\
\left. \cdots \to \bigoplus_{x \in X^{n-1}}(j_x)_*\left(\mck^{\Nis}_{1,\left\{x\right\}}\right) \to \bigoplus_{x \in X^n}(j_x)_*\left(\mck^{\Nis}_{0,\left\{x\right\}}\right) \to 0 \right)
\end{multline*}
so that the global sections $\Gamma(X,\dumcg^{\cdot}_{n,X})$ simply recover the usual Gersten complex $\mcg_n^{\cdot}(X)$ for either topology.

\begin{lem}\label{hypercohomology}The hypercohomology of $\dumcg^{\cdot}_{n,X}$ may be identified with the usual cohomology of the Gersten complex:
\begin{enumerate}
\item[(i)] For all $p \geq 0$, $\mathbb{H}^p_{\Zar}(X,\dumcg_{n,X}) \cong H^p(\mcg_n(X))$.
\item[(ii)] For all $p \geq 0$, $\mathbb{H}^p_{\Nis}(X,\dumcg_{n,X}) \cong H^p(\mcg_n(X))$.
\end{enumerate}
\end{lem}
\begin{proof}We shall only prove (ii) as the proof of (i) is virtually identical. It suffices to show that each term $\ds \bigoplus_{x \in X^{m}}(j_x)_*\left(\mck^{\Nis}_{n-m,\left\{x\right\}}\right)$ has no higher Nisnevich cohomology. Fix $x \in X^m$ and $j_x: \left\{x\right\}=\Spec k(x) \to X$. For any \'{e}tale $U \to \Spec k(x)$, $H^p_{\Nis}(U,\mcf) = 0$ for all sheaves $\mcf$ on $(\Spec k(x))_{\Nis}$ and all $p > \dim U = 0$ by \cite[12.2]{Mazza}. Thus, $R^p(j_x)_* = 0$ for $p > 0$ by \cite[III.1.13]{Milne}. From the Leray spectral sequence, we have
\[ E_2^{pq} = H^p_{\Nis}\left(X,R^p(j_x)_*\left(\mck^{\Nis}_{n-m,\left\{x\right\}}\right)\right) \Rightarrow H^{p+q}_{\Nis}\left(\Spec k(x),\mck^{\Nis}_{n-m,\left\{x\right\}}\right) \]
which degenerates to give isomorphisms
\[ H^p_{\Nis}\left(X,(j_x)_*\left(\mck^{\Nis}_{n-m,\left\{x\right\}}\right)\right) \cong H^{p}_{\Nis}\left(\Spec k(x),\mck^{\Nis}_{n-m,\left\{x\right\}}\right) = 0 \mbox{ for $p > 0$.} \]
Since $X$ is Noetherian, cohomology commutes with direct sums \cite[I.3.11.1]{Tamme}, and thus, each term of $\dumcg_{n,X}$ is acyclic for $\Gamma_{\Nis}(X,-)$.
The claim now follows since the hypercohomology spectral sequence
\[ E_2^{pq} = H^p\left(H^q_{\Nis}\left(\dumcg_{n,X}^{\cdot}\right)\right) \Rightarrow \mathbb{H}^{p+q}_{\Nis}\left(X,\dumcg_{n,X}^{\cdot}\right) \]
degenerates to give the desired isomorphism.
\end{proof}

We now show that the stalks of $\dumcg_{n,X}$ allow us to recover $\mcg(\mco_{X,x})$ or $\mcg(\mco_{X,x}^h)$ depending on whether we are working in the Zariski or Nisnevich topologies.

\begin{lem}\label{stupid_hensel}Let $X$ be a Noetherian scheme and let $x,x' \in X$. If $x$ lies in the Zariski closure of $\left\{x'\right\}$, then $\Spec k(x') \times_X \Spec \mco_{X,x}^h = \coprod_{i=1}^m \Spec k(y_i)$ where $y_i$ are precisely the points of $\Spec \mco_{X,x}^h$ lying over $x'$.
\end{lem}
\begin{proof}First, note that $\Spec k(x') \times_X \Spec \mco_{X,x}^h$ is faithfully flat over $\Spec k(x') \times_X \Spec \mco_{X,x}$, and the latter is empty if and only if $x \notin \overline{\left\{x'\right\}}$. We henceforth assume that $\Spec k(x') \to X$ factors through $\Spec \mco_{X,x}$, and so, $\Spec k(x') \times_X \Spec \mco_{X,x}^h = \Spec k(x') \times_{\Spec \mco_{X,x}} \Spec \mco_{X,x}^h$. 
If we let $A = \mco_{X,x}$, $x'$ corresponds to a prime ideal $\fp \subset A$, so $k(x')$ is just the fraction field $E$ of $A/\fp$. Since $A/\fp \otimes_A A^h = (A/\fp)^h$, we are reduced to proving that if $A$ is a Noetherian local domain with fraction field $E$, then $E \otimes_A A^h$ is a finite product of fields $L_i$, each of which is separable algebraic over $E$. We can write $A^h = \fcolim B_{\lambda}$ where each $B_{\lambda}$ is \'{e}tale over $A$. Thus, each $E \otimes_A B_{\lambda}$ is finite over $E$ (since \'{e}tale morphisms are quasifinite). Thus $E \to E \otimes_A A^h = \fcolim (E \otimes_A B_{\lambda})$ is an integral extension of $E$ and so must be zero-dimensional. On the other hand, $E \otimes_A A^h$ is Noetherian, being a localization of a $A^h$, meaning that $E \otimes_A A^h$ is a product of Artin local rings. Since $E \to E \otimes_A A^h$ must also be geometrically regular (Example \ref{geomregexmps}(a)), the rings are reduced, and the claim follows.
\end{proof}

\begin{lem}\label{complex_stalk}Let $X$ be a Noetherian scheme and consider the Gersten complex $\dumcg_{n,X}^{\cdot}$ on $X_{\Nis}$. Fix $x \in X$.
\begin{enumerate}
\item[(i)] In the Zariski topology, the stalk $\left(\dumcg_{n,X}^{\cdot}\right)_x$ is $\mcg_n^{\cdot}(\mco_{X,x})$, the Gersten complex on $\mco_{X,x}$.
\item[(ii)] In the Nisnevich topology, the stalk $\left(\dumcg_{n,X}^{\cdot}\right)_x$ is $\mcg_{n}^{\cdot}(\mco_{X,x}^h)$, the Gersten complex on $\mco_{X,x}^h$.
\end{enumerate}
\end{lem}
\begin{proof}We begin by noting that for any field $E$, the presheaf $U \mapsto K_n(U)$ is already a sheaf on $(\Spec E)_{\Nis}$. This simply owes to the fact that every object $U$ of$(\Spec E)_{\Nis}$ is a finite disjoint union $\coprod \Spec L_i$ with each $L_i/E$ finite and separable and the fact that for any Nisnevich cover $\left\{V_i \to U\right\}$ each $\Spec L_i$ will appear as a connected component of $V_j$ for some $j$.

We shall only prove (ii) since the argument (i) follows along the same lines and, at any rate, already appears in the proof of \cite[VII.5.8]{Quillen}. Fix $x' \in X^p$. We shall analyze the Nisnevich stalk of each $(i_{x'})_*(\mck_{q,\left\{x'\right\}}^{\Nis})$ at $x$. Given $x \in X$, \cite[II.3.2(b)]{Milne} shows that
\[ \left((i_{x'})_*\left(\mck_{q,\left\{x'\right\}}^{\Nis}\right)\right)_x = \varinjlim_{U}\mck_{q,\left\{x'\right\}}^{\Nis}(\Spec k(x') \times_X U) = \varinjlim_{U}K_q(\Spec k(x') \times_X U) \]
where the limit is taken over all \'{e}tale $U \to X$ through which $\Spec k(x) \to X$ factors. By \cite[VII.2.2]{Quillen} the rightmost object is just $K_q(\Spec k(x') \times_X \Spec(\mco_{X,x}^h))$. But by Lemma \ref{stupid_hensel}, $\Spec k(x') \times_X \Spec \mco_{X,x}^h = \coprod_{i=1}^m \Spec k(y_i)$ where $y_i \in \Spec \mco_{X,x}^h$ are the points (necessarily of codimension $p$) lying over $x'$. Since taking stalks commutes with direct sums we get that $\left(\dumcg_{n,X}^{\cdot}\right)_x = \mcg_n^{\cdot}(\mco_{X,x}^h)$.
\end{proof}
 
\subsection{Local Flatness Criteria} For the sake of convenience, we shall state here two corollaries of the local flatness criterion that we shall need to use repeatedly.
\begin{lem}\label{localflatness1}\cite[Thm. 22.5+Cor.]{Matsumura} Let $(R,\fm_R,k) \to (S,\fm_S,k')$ be a local homomorphism of Noetherian local rings. Let $x_1, \cdots, x_n \in \fm_S$. The following are equivalent:
\begin{enumerate}
\item[(a)] $x_1, \cdots, x_n$ is a regular sequence on $S$ and $R \to S/(x_1, \cdots x_n)$ is flat.
\item[(b)] $x_1, \cdots, x_n$ is a regular sequence on $S \otimes_R k$ and $R \to S$ is flat.
\end{enumerate}
\end{lem}

\begin{lem}\label{localflatness2}\cite[7.4]{Swan} Let $R \to S \to T$ be maps of Noetherian rings. Let $I \subset R$ be an ideal such that $IS$ is contained in the Jacobson radical of $T$. If $R \to T$ and $S/IS \to T/IT$ are flat then so is $S \to T$. 
\end{lem}
\subsection{Smoothness, Formal-Smoothness, and Geometric Regularity}\label{formal_smoothness}
Here, we recall some of the basic notions of formal smoothness and its relationship with ``genuine'' smoothness.

\begin{defn}Let $R \to A$ be a morphism of rings and let $I \subset A$ be an ideal. We say that $A$ is $I$-smooth over $R$ if the following condition holds:
For all rings $C$, all nilpotent ideals $J \subset C$, and all commutative diagrams
\[
\xymatrix{
R \ar[r] \ar[d] & C \ar@{>>}[d] \\
A \ar^{\psi}[r] &       C/J
}\]
with $C/J$ discrete and $\psi$ $I$-adically continuous (i.e. $\psi(I^m) = 0$ for $m >> 0$), there exists a lift $\widetilde{\psi}:A \to C$ making the diagram commute.
\end{defn}

We remark that if $I' \subset I \subset A$, then since the $I'$-topology is finer than that defined by $I$, the condition of $I'$-smoothness implies $I$-smoothness.

\begin{defn}\label{g_reg} Let $X$, $Y$, and $Z$ be Noetherian schemes.
\begin{enumerate}
\item[(a)] If $X$ is defined over a field $k$, we say that $X \to \Spec k$ is \emph{geometrically regular} if for all finite field extensions $L/k$, $X_L := X \times_k \Spec L$ is a regular scheme.
\item[(b)] A morphism $Y \to Z$ of Noetherian schemes is called \emph{geometrically regular} if it is flat and for all $z \in Z$, the fibre $Y_z := Y \times_Z \Spec k(z) \to \Spec k(z)$ is geometrically regular in the sense of (a). If $Y \to Z$ is geometrically regular and of finite-type, it is called \emph{smooth}.
\end{enumerate}
\end{defn}

\begin{rem}A morphism $Y \to Z$ satisfying the conditions of (b) in Definition \ref{g_reg} is often called ``regular'' while ``geometrically regular'' is reserved exclusively for maps to $\Spec k$ ($k$ a field). In this paper, we shall follow the Swan's convention \cite{Swan} and use ``geometrically regular'' for both notions. We do this to emphasize that definition (a) is indeed a special case of (b) and also to eliminate confusion between a relative and absolute notion: being ``geometrically regular'' is a property of a morphism while being ``regular'' is an intrinsic property of a ring or scheme.
\end{rem}

In the local case, the relationship between formal smoothness and geometric regularity is afforded by the following important lemma:

\begin{lem}\label{fsmooth_fibre}\cite[2.6.5]{Majadas} Let $(R,\fm_R,k) \to (S,\fm_S,k')$ be local homomorphism of Noetherian local rings. The following are then equivalent:
\begin{enumerate}
\item[(a)] $R \to S$ is $\fm_S$-smooth.
\item[(b)] $R \to S$ is flat and $k \to S \otimes_R k$ is geometrically regular.
\end{enumerate}
\end{lem}

\begin{exmps}\label{geomregexmps}While all smooth morphisms are geometrically regular, the non-finite-type case arises fairly often in local algebra. Here are some examples that are relevant for this paper.
\begin{enumerate}
\item[(a)] If $(R,\fm_R,k)$ is a Noetherian local ring, then the Henselization map $R \to R^h$ is geometrically regular by virtue of being a filtered colimit of essentially-\'{e}tale $R$-algebras.
\item[(b)] If a Noetherian local ring $(R,\fm_R,k)$ is excellent, the completion map $R \to \widehat{R}$ is geometrically regular. Such rings are ubiquitous in algebraic geometry.
\item[(c)] If $A$ is excellent, a flat local map of Noetherian local rings $(R,\fm_R,k) \to (A,\fm_A,L)$ is geometrically regular if and only if $k \to A \otimes_R k$ is. This is Andr\'{e}'s so-called theorem of localization \cite{Andre}.
\item[(d)] A mixed characteristic regular local ring $(A,\fm_A,L)$ is called \emph{unramified} if there is some prime number $p \in \fm_A-\fm_A^2$. Then $\mathbb{Z}_{(p)} \to A$ is geometrically regular since the residue fields of $\mathbb{Z}_{(p)}$ are perfect and the fibre rings are regular.
\item[(e)] If $k$ is a perfect field, then every field extension $k \to L$ is geometrically regular.
\end{enumerate}
\end{exmps}

It is a well-known fact that a section of a smooth morphism $X \to Y$ is a closed immersion that is locally cut out by a regular sequence on $X$. We present here the local analogue, which we shall need later:

\begin{lem}\label{section}Let $(R,\fm_R,k)$ and $(S,\fm_S,k')$ be Noetherian local rings. Suppose that $\phi: R \to S$ is an $\fm_S$-smooth local homomorphism. Let $r: S \to R$ be a retraction in the sense that $r \circ \phi = \operatorname{Id}_R$. Then $I = \ker(r)$ is generated by an $S$-regular sequence of length $d = \dim S - \dim R$.
\end{lem}
\begin{proof}By Lemma \ref{fsmooth_fibre}, $R \to S$ is flat with geometrically regular fibre $k \to S/\fm_R S$. In particular, $S/\fm_R S$ is a regular local ring. If we apply $- \otimes_R k$ to the short-exact sequence
\[ 0 \to I \to S \stackrel{r}{\longrightarrow} R \to 0 \]
we obtain
\[ 0 \to I/\fm_R I \to S/\fm_R S \to k \to 0. \]
$I/\fm_R I$ is the maximal ideal of $S/\fm_R S$ and is therefore generated by a regular sequence of length $d = \dim(S/\fm_R S) = \dim S - \dim R$ \cite[15.1]{Matsumura}. By Nakayama's Lemma, $I$ is generated by $d$ elements $t_1, \cdots, t_d$ of $S$. As $R \to S$ is flat and the $t_i$ form an $S/\fm_R S$-regular sequence, we conclude by  Lemma \ref{localflatness1} that the $t_i$ form a regular sequence on $S$.
\end{proof}

We close this subsection by introducing Popescu's Theorem, which states that every geometrically regular morphism arises as a filtered colimit of smooth ones.

\begin{thm}[Popescu's Theorem]\cite{Popescu}, \cite{Swan}\label{popescu} Let $R \to A$ be a geometrically regular map of Noetherian rings. Then there exists a filtered system $\left\{A_{\lambda}\right\}_{\lambda \in \Lambda}$ of smooth (and hence finite-type) $R$-algebras such that $A = \fcolim A_{\lambda}$.
\end{thm}

\subsection{K-Theoretic Consequences of Popescu's Theorem}\label{popescu_consequence} Quillen originally proved Gersten's Conjecture for local rings essentially smooth over a field \cite[VII.5.11]{Quillen}. If $(A,\fm_A,L)$ is an equicharacteristic regular local ring, we can take a perfect subfield $k_{\circ} \subset A$, thereby forcing $k_{\circ} \to A$ to be geometrically regular. By Popescu's Theorem, we can therefore write $A$ as a filtered colimit of local rings $\left\{A_{\lambda}\right\}_{\lambda \in \Lambda}$ which are essentially smooth over $k_{\circ}$. However, the transition maps $A_{\lambda} \to A_{\mu}$ need not be flat and therefore are not compatible with the codimensional filtration. Panin's \cite{Panin} clever insight was to realize the Gersten cohomology groups $H^p(\mcg_n^{\cdot}(A))$ via Zariski sheaf cohomology, an invariant which, by Grothendieck's Limit Theorem \cite[VII.5.7]{SGA4.2}, is compatible with inverse limits of affine schemes, regardless of whether the transition maps are flat.  We state here the key results that we will need in the proof of our Main Theorem \ref{mainthm}.
\begin{thm}\cite{Panin}\label{panin} If $A$ is a regular local ring containing a field, then $A$ satisfies Gersten's Conjecture. That is,
\[ H^p(\mcg^{\cdot}_n(A)) = \left\{ \begin{array}{ll} K_n(A) & p = 0 \\ 0 & p > 0 \end{array} \right. \]
\end{thm}

This result, in conjunction with \cite[VII.5.8]{Quillen}, yields the following useful corollary, which permits us to realize cohomology of the Gersten complex as Zariski sheaf cohomology:

\begin{cor}\label{panin_sheaf} Let $k$ be a field and suppose that $X$ is a regular, Noetherian scheme defined over $k$. Denote by $\mck_n$ the Zariski sheafification of the presheaf $U \mapsto K_n(U)$. Then there are isomorphisms
\[ H^p_{\Zar}(X,\mck_n) \stackrel{\sim}{\longrightarrow} H^p(\mcg^{\cdot}_n(X)). \]
\end{cor}

Since the $K$-theory of vector bundles commutes with inverse limits of Noetherian affine schemes \cite[VII.2.2]{Quillen}, the Grothendieck Limit Theorem \cite[VII.5.7]{SGA4.2} (see also \cite[\S6]{Panin}) furnishes the following result:

\begin{prop}\label{panin_limit} Let $\Lambda$ be a directed set and let $\left\{X_{\lambda}\right\}_{\lambda \in \Lambda}$ be an inverse system of Noetherian affine schemes such that the limit $\ds X := \varprojlim_{\lambda \in \Lambda} X_{\lambda}$ is Noetherian. Then for all $p,n \geq 0$ there are isomorphisms 
\[\varinjlim_{\lambda \in \Lambda}H^p_{\Zar}(X_\lambda, \mck_n) \stackrel{\sim}{\longrightarrow} H^p_{\Zar}(X,\mck_n). \]
\end{prop}

\subsection{Chow Groups and Codimension}\label{codimension}For an arbitrary Noetherian scheme $X$, there is some ambiguity regarding the definition of the Chow group. We adopt the convention in \cite[8.1]{GilletSoule}. We define $Z^p(X)$ to be the free abelian group on integral subschemes $V \subset X$ of codimension $p$ (or, equivalently, the free abelian group on the set $X^p = \left\{x \in X: \dim \mco_{X,x} = p \right\}$). For an integral, closed subscheme $W$ of codimension $p-1$, and an $f \in k(W)^{\times}$, one defines the rational cycle $\operatorname{div}(f) = \sum_{V}\ord_V(f)[V] \in Z^p(X)$ where the sum is taken over all codimension-$p$ integral subschemes $V \subset X$ that are contained in $W$ and $\ord_V$ is the order function on the $1$-dimensional domain $\mco_{W,V}$ (cf. \cite[A.3]{Fulton}). In this way, one obtains a map
\[ \bigoplus_{\mathclap{\codim W = p-1}}k(W)^{\times} \stackrel{\operatorname{div}}{\longrightarrow} Z^p(X) \]
whose cokernel is defined to be $\ch^p(X)$. The map $\operatorname{div}$ agrees up to sign with the $E_1^{p-1,-p} \to E_1^{p,-p}$ differential in the coniveau spectral sequence on $X$:

\begin{lem}\label{gersten_chow} \cite[VII.5.14]{Quillen} \cite[V.9.5]{Kbook} Let $X$ be a Noetherian scheme. Then the following identifications hold for $p \geq 0$:
\[ \ch^p(X) = E_2^{p,-p}(X) = H^p(\mcg_p^{\cdot}(X)). \]
Here, $E_r^{pq}(X)$ is the coniveau spectral sequence on $X$ (\ref{BGQ}).
\end{lem}

This definition differs slightly from Fulton's \cite[1.3]{Fulton} where the rational cycle associated to $f \in k(W)^{\times}$ is $\operatorname{div}(f) = \sum_{V}\ord_V(f)[V]$ and the sum is now taken over all $V$ having codimension $1$ in $W$. In the geometric setting, the two notions are equivalent. However, if $X$ is not catenary, codimension may not be additive; so this proposed rational cycle may not even be homogeneous. The following lemma gives sufficient conditions under which the two definitions agree.

\begin{lem}\label{good_codim}Let $X$ be a Noetherian scheme and suppose that
\begin{enumerate}
\item[(a)] $X$ is catenary.
\item[(b)] For all $x \in X$, $\mco_{X,x}$ is equidimensional.
\end{enumerate}
Then if $Y \subset X$ is a closed, irreducible subscheme of codimension $r$ and $x \in Y^s$, then $x \in X^{r+s}$. The first definition of rational equivalence therefore agrees with Fulton's.
\end{lem}
\begin{proof}By hypothesis, the generic point of $Y$ corresponds to a height $r$ prime $\fp$ of $\mco_{X,x}$. Choose $\fq_0 \subset \fp$ to be a minimal prime of $\mco_{X,x}$. Since $\mco_{X,x}$ is equidimensional, $\dim \mco_{X,x} = \dim \mco_{X,x}/\fq_0$. As $(\mco_{X,x})_{\fp}$ is also equidimensional, there exists a saturated chain $\fq_0 \subset \fq_1 \subset \cdots \subset \fq_r = \fp$ of distinct primes in $\mco_{X,x}$. On the other hand, $\mco_{Y,x} = \mco_{X,x}/\fp$, so we can get a saturated chain of primes $\fp = \fp_0 \subset \fp_1 \subset \cdots \subset \fp_s = \fn$ where $\fn$ is the maximal ideal of $\mco_{X,x}$. Concatenating these chains together gives a saturated chain of $r+s$ primes between $\fq$ and $\fn$. Since $\mco_{X,x}$ is catenary, all saturated chains between $\fq_0$ and $\fn$ must have length $r+s$, whence $\dim \mco_{X,x} = r+s$ as claimed.
\end{proof}

\begin{rem}A Noetherian scheme $X$ will satisfy the conditions of Lemma \ref{good_codim} if it is Cohen-Macaulay \cite[17.6, 17.9]{Matsumura}.
\end{rem}

\section{The Finite-Type Case}\label{finite_case}
\subsection{Structure Theorems} Quillen's proof of Gersten's Conjecture in the geometric case relies on a type of Noether Normalization Lemma \cite[VII.5.12]{Quillen}. Here, we prove local analogues in the mixed characteristic setting. In the next section, we will use these results in conjunction with Popescu's Theorem to bootstrap up to the main theorem.
\begin{lem}\label{local_finiteness}Let $R \to S$ be a finite map. Let $\fq \in \Spec S$ be a prime contracting to $\fp \in \Spec R$. Then $k(\fp) = R_{\fp}/\fp R_{\fp} \to S_{\fq}/\fp S_{\fq}$ is finite.
\end{lem}
\begin{proof}If we base-change to $k(\fp)$, we see that $k(\fp) \to S \otimes_R k(\fp)$ is finite, meaning that $S \otimes_R k(\fp)$ is a product of artin local rings. $S_{\fq}/\fp S_{\fq}$ is merely one of the factors and hence must also be finite over $k(\fp)$.
\end{proof}

The first of our structure theorems is listed below. We remark that when we say a ring $S$ \emph{essentially} has property $\mathcal{P}$ over $R$ (where $\mathcal{P}$ could be finite-type, finite, smooth, etc.), we shall mean that $S$ is the localization of an $R$-algebra $S'$ having property $\mathcal{P}$.

\begin{lem}\label{structure_zariski}Let $(R,\fm_R, k) \to (A,\fm_A,L)$ be an essentially smooth local homomorphism of Noetherian local rings. Fix $f \in \fm_{A} - \fm_R A$. Then there exists a local $R$-algebra $(B,\fm_B,K)$ such that
\begin{enumerate}
\item $B \to A$ is $\fm_A$-smooth.
\item $B \to A/fA$ is flat and essentially finite.
\item $\dim B = \dim A/fA = \dim A - 1$.
\end{enumerate} 
\end{lem}
\begin{proof}By assumption, $A = A'_{\fq}$ where $A'$ is a smooth (and hence finite-type) $R$-algebra and $\fq \in \Spec A'$. For ease of notation, we shall write $\overline{S} := S/\fm_R S$ for any $R$-algebra $S$. After base-changing to $k$, the residue field of $R$, we have a smooth map $k \to \overline{A'}$. Since $\overline{A'}$ is regular (and hence a product of domains), we may assume --- after replacing $A'$ by $A'_g$ for some $g \notin \fq$ --- that $\overline{A'}$ is a domain. By assumption, $f$ is a non-zerodivisor on $\overline{A'}$, so we may invoke Quillen's normalization lemma \cite[VII.5.12]{Quillen} to obtain a morphism $\phi_{\circ}: B_{\circ} := k[X_1, \cdots, X_{r-1}] \hookrightarrow \overline{A'}$ where $r = \dim \overline{A'}$, $B_{\circ} \to \overline{A'}$ is smooth at $\fq$, and the composite $B_{\circ} \to \overline{A'} \to \overline{A'}/f\overline{A'}$ is finite.

To show that $B_{\circ} \to \overline{A'}/f \overline{A'}$ is flat, it suffices to check locally at maximal ideals of $\overline{A'}/f \overline{A'}$. Fix a maximal ideal $\fn$ of $\overline{A'}/f \overline{A'}$. Since $\overline{A'}$ is a finite-type $k$-domain, all of its maximal ideals have height $r = \dim(\overline{A'})$ \cite[15.6]{Matsumura}; whence $\dim(\overline{A'}/f \overline{A'})_{\fn} = r-1$. Since $B_{\circ} \to \overline{A'}/f \overline{A'}$ is finite, $\fn$ contracts to a maximal prime $\fn_{\circ}$ of $B_{\circ}$, which must therefore also have height $r-1$. Since $B_{\circ} \to \overline{A'}/f \overline{A'}$ is finite, the image of $\fn_{\circ}$ in $(\overline{A'}/f \overline{A'})_{\fn}$ is $\fn$-primary by Lemma \ref{local_finiteness}.  We conclude that $(B_{\circ})_{\fn_{\circ}} \to (\overline{A'}/f \overline{A'})_{\fn}$ flat as its source is regular and target is Cohen-Macaulay \cite[23.1]{Matsumura}.

Put $B' = R[X_1, \cdots, X_{r-1}]$ and define an $R$-algebra map $\phi: B' \to A'$ by mapping the $X_i$ to lifts of the $\phi_{\circ}(X_i) \in \overline{A'}$. Set $\fp = \phi^{-1}(\fq)$ and define the local ring $(B,\fm_B,K)$ to be $B'_{\fp}$. We shall now verify conditions (1)-(3) for $B \to A$.

\textbf{Claim 1:} ($B \to A$ is $\fm_A$-smooth.) First, note that $\overline{B} = (B_{\circ})_{\fp} \to (\overline{A'})_{\fq} = \overline{A}$ is essentially smooth (and hence flat). Since $R \to B \to A$ is flat, we can conclude that $B \to A$ is flat by Lemma \ref{localflatness2}. Moreover, the fibre $K \to A \otimes_B K = \overline{A} \otimes_{\overline{B}} K$ must be geometrically regular, so by Lemma \ref{fsmooth_fibre}, we have the claim.

\textbf{Claim 2:} ($B \to A/fA$ is flat and essentially finite.) Since $R \to A$ is flat and $f$ is a non-zerodivisor on the regular local ring $\overline{A}$, Lemma \ref{localflatness1} says that $R \to A/fA$ is flat and $f$ is a non-zerodivisor on $A$. We already saw that $\overline{B} = (B_{\circ})_{\fp} \to (\overline{A'}f \overline{A'})_{\fq} = \overline{A}/f\overline{A}$ is flat, so $B \to A/fA$ is flat by Lemma \ref{localflatness2}. As $B \to A/fA$ is essentially of finite-type, to show essential finiteness, it will suffice by Peskine's version of Zariski's Main Theorem \cite{Peskine}, \cite[9.1]{Swan}, to show that $K \to A/fA \otimes_B K = (\overline{A'}/f \overline{A'})_{\fq} \otimes_{(B_{\circ})_{\fp}} K$ is finite. This then follows from the finiteness of $B_{\circ} \to \overline{A'}/f\overline{A'}$ and Lemma \ref{local_finiteness}.

\textbf{Claim 3:} ($\dim B = \dim A/fA = \dim A - 1$.) Since $f \in A$ is a non-zerodivisor, we have that $\dim A/fA = \dim A - 1$. As $B \to A/fA$ is flat and essentially finite, both rings must have the same dimension (cf. \cite[15.1]{Matsumura}). 
\end{proof}

If we Henselize, we can arrange for actual finiteness instead of essential finiteness:

\begin{lem}\label{structure_nisnevich}Let $(R,\fm_R, k) \to (A,\fm_A,L)$ be an essentially smooth local homomorphism of Noetherian local rings. Fix $f \in \fm_{A^h} - \fm_R A^h$. Then there exists a local Henselian $R$-algebra $(C,\fm_C,K)$ such that 
\begin{enumerate}
\item $C \to A^h$ is $\fm_{A^h}$-smooth.
\item $C \to A^h/fA^h$ is finite and flat.
\item $\dim C = \dim A^h/fA^h = \dim A - 1$.
\end{enumerate}
\end{lem}
\begin{proof}We can write $A^h = \fcolim A_{\lambda}$ where $\left\{A_\lambda\right\}_{\lambda \in \Lambda}$ is a filtered system of local, essentially \'{e}tale $A$-algebras with residue field $L$ (see, for example, \cite[\S I.4]{Milne}). We must therefore have $f \in A_{\lambda} - \fm_{R}A_{\lambda}$ for some $\lambda$; so we might as well assume that $f \in \fm_A - \fm_R A$ to begin with. Let $(B, \fm_B, K) \to (A, \fm_A, L)$ be as in Lemma \ref{structure_zariski}. We therefore obtain a natural map $B^h \to A^h$; put $(C, \fm_C, K) = (B^h, \fm_B B^h, K)$.

\textbf{Claim 1:} ($C \to A^h$ is $\fm_{A^h}$-smooth.) $B \to A \to A^h$ is obviously $\fm_{A^h}$-smooth. If we look at $B \to C \to A^h$, then by using Lemma \ref{localflatness2}, the flatness of $C \to A^h$ follows from the flatness of $K = C/\fm_B C \to A^h/\fm_B A^h$. Since $B \to A^h$ and $C = B^h \to A^h$ have the same special fibre, we can conclude that $C \to A^h$ is $\fm_{A^h}$-smooth by Lemma \ref{fsmooth_fibre}.

\textbf{Claim 2:} ($C \to A^h/fA^h$ is finite and flat.) For flatness, we appeal to Lemma \ref{localflatness2} and the observations that $B \to A/fA \to (A/fA)^h = A^h/fA^h$ and $K = C/\fm_B C \to (A^h/fA^h)/\fm_B (A^h/fA^h)$ are both flat. Next, note that $B \to A/fA$ is essentially finite; that is, there exists a finite (and hence semilocal) $B$-algebra $D$ with $A/fA = D_{\fn}$ for some maximal ideal $\fn$ of $D$. Take $B \to D$ and base-change to $C = B^h$. $D \otimes_B C$ is finite over the Henselian ring $C$ and so factors into a product of Henselian local rings that are finite over $C$ \cite[I.4.2(b)]{Milne}. Since the special fibre of $B \to C$ is an isomorphism, the maximal ideals of $D$ and $D \otimes_{B} C$ are in bijective correspondence. Choose the unique maximal ideal $\fn'$ of $D \otimes_B C$ lying over $\fn$. Since $(D \otimes_{B} C)_{\fn'}$ is simply a projection onto one of the factors, the composite $C \to D \otimes_B C \to (D \otimes_B C)_{\fn'}$ is finite. On the other hand, by the universal property of Henselization, the flat map $A/fA = D_{\fn} \to (D \otimes_B C)_{\fn'}$ extends uniquely to $A^h/fA^h = (A/fA)^h \to (D \otimes_B C)_{\fn'}$. By Lemma \ref{localflatness2}, $A^h/fA^h \to (D \otimes_B C)_{\fn'}$ is faithfully flat and therefore injective; therefore, $C \to A^h/fA^h$ must be finite as the composite $C \to A^h/fA^h \to (D \otimes_B C)_{\fn'}$ is.

\textbf{Claim 3:} ($\dim C = \dim A^h/fA^h = \dim A - 1$.) As $f$ is a non-zerodivisor on $A$ (see the proof of Lemma \ref{structure_zariski}), it must also be a non-zerodivisor on $A^h$. The claim now immediately follows from the preceding two claims.
\end{proof}

In the complete case we can obtain an analogue of Lemma \ref{structure_nisnevich} without any finite-type hypotheses.
\begin{lem}\label{structure_complete}Let $(R, \fm_R, k) \to (A, \fm_A, L)$ be a flat map of Noetherian local rings with $A$ complete. Suppose that $k \to L$ is geometrically regular and that $A/\fm_R A$ is regular. Fix $f \in A - \fm_R A$ Then there exists a local $R$-subalgebra $(B, \fm_B, L)$ of $A$ such that 
\begin{enumerate}
\item $B \to A$ is $\fm_A$-smooth.
\item $B \to A/fA$ is finite and flat.
\item $\dim B = \dim A/fA = \dim A - 1$.
\end{enumerate}
\end{lem}
\begin{proof}By \cite[10.3.1]{EGAIII}, we may construct an over-ring $(R',\fm_{R'},L)$ of $R$ such that $R \stackrel{u}{\longrightarrow} R'$ is flat and $R' \otimes_R k \cong L$ as $k$-algebras. We may also take $R'$ to be complete. Using Lemma \ref{fsmooth_fibre}, we see that $R \to R'$ is $\fm_{R'}$-smooth. If we define $\phi_n$ to be the composite $R \to A \to A/\fm_A^n$ and let $\psi_1:R' \to R'/\fm_{R'} = L = A/\fm_A$ be the canonical quotient map, we then get $\phi_1 = \psi_1 \circ u$. Assuming that we have inductively constructed $\psi_n: R' \to A/\fm_A^n$ so that $\phi_n = \psi_n \circ u$, invoking the formal-smoothness of $u$ gives a lift
\[
\xymatrix{
R \ar^{\phi_{n+1}}[r] \ar_{u}[d]  & A/\fm_A^{n+1} \ar[d] \\
R' \ar_{\psi_{n}}[r]     \ar^{\psi_{n+1}}@{.>}[ru]  & A/\fm_A^n 
}.
\]
Since $A$ is complete, we therefore obtain an $R$-algebra map $\ds R' \to \varprojlim_{n}A/\fm_A^n = A$. Note that $\fm_{R'}A = \fm_{R}A$, so $\overline{A} := A/\fm_{R'}A$ is a regular local ring. By assumption, $f$ is nonzero on $A/\fm_{R'}A$ and so is a parameter. If $d = \dim \overline{A}$, we can use a prime-avoidance argument (see, for example, \cite[2.6]{Koszul}) and choose elements $g_1, \cdots, g_{d-1} \in A$ such that $\overline{A}/(g_1, \cdots, g_{d-1})\overline{A}$ is regular with $f,g_1, \cdots, g_{d-1}$ a regular sequence on $\overline{A}$. Define $B := R'[[X_1, \cdots X_{d-1}]] \to A$ by mapping $X_i$ to $g_i$.

\textbf{Claim 1:} ($B \to A$ is $\fm_A$-smooth.) Consider $R \to B \to A$. Base-changing to $k = R/\fm_R$ gives $k \to \overline{B} = L[[X_1, \cdots, X_{d-1}]] \to \overline{A}$. $\overline{B} \to \overline{A}$ is flat since both rings are regular and the $X_i$ map to a regular sequence on $\overline{A}$ \cite[23.1]{Matsumura}. Therefore, $B \to A$ is flat by Lemma \ref{localflatness2}. The special fibre of $B \to A$ is $L \to A/\fm_{B}A = \overline{A}/(g_1, \cdots, g_{d-1})\overline{A}$. Since $A/\fm_{B}A$ is regular with residue field $L$, the map $L \to A/\fm_{B}$ is geometrically regular by \cite[\S 28, Lem. 1]{Matsumura}; and hence, we may conclude that $B \to A$ is $\fm_{A}$-smooth by Lemma \ref{fsmooth_fibre}.

\textbf{Claim 2:} ($B \to A/fA$ is finite and flat.) Note $A/fA \otimes_B L = \overline{A}/(f,g_1, \cdots, g_{d-1})\overline{A}$ is zero-dimensional and hence finite over $L = B/\fm_B$ (since $A$ and $B$ share the same residue field.) Since both rings are also complete, $B \to A/fA$ is finite by \cite[8.4]{Matsumura}. Since $f$ is a non-zerodivisor on $\overline{A}$ and $R \to A$ is flat, Lemma \ref{localflatness1} says that $R \to B \to A/fA$ is flat and $f$ is regular on $A$. On the other hand, $\overline{B} \to \overline{A}/f\overline{A}$ is flat since $g_1, \cdots g_{d-1}$ is a regular sequence on the CM ring $\overline{A}/f \overline{A}$; so by Lemma \ref{localflatness2}, $B \to A/fA$ is flat.

\textbf{Claim 3:} ($\dim B = \dim A/fA = \dim A - 1$.) This follows at once from the fact that $f$ is regular on $A$ and the finite flatness of $B \to A/fA$.
\end{proof}

\subsection{Local Triviality of Coniveau} Using our normalization lemmas, we are now poised to prove a special case of our main theorem:
\begin{prop}\label{main_thm_weak}Let $(R,\fm_R,k)$ be a Noetherian local ring and let $(A,\fm_A,L)$ be a local Noetherian $R$-algebra. Fix some $f \in \fm_A - \fm_R A$ and suppose that \textbf{either}:
\begin{enumerate}
\item[(a)] $A$ is the Henselization of a local, essentially smooth $R$-algebra.
\item[(b)] $A$ is complete and $\fm_A$-smooth over $R$ with $k \to L$ geometrically regular.
\end{enumerate}
Then the natural inclusion $\mathbf{i}:\mcm^p(A/fA) \to \mcm^p(A)$ induces the zero map on $K$-theory for all $p \geq 0$.
\end{prop}
\begin{proof}In either case, Lemma \ref{structure_nisnevich} or \ref{structure_complete} will give us a local $R$-algebra $(B,\fm_B,K)$ and a map $B \to A$ such the following hold:
\begin{enumerate}
\item $B \to A$ is $\fm_A$-smooth.
\item $B \to A/fA$ is finite and flat.
\item $\dim B = \dim A/fA = \dim A - 1$.
\end{enumerate}
Let $D = A \otimes_{B} A/fA$. By base-change, $D$ is $\fm_A D$-smooth over $A/fA$ as well as flat and finite over $A$. The map $a \otimes b \mapsto ab$ gives a surjection $r:D \to A/fA$. Let $I$ be its kernel. We claim that $I \cong D$ as modules over $D$. Since $A/fA$ is local, $I$ is contained in a unique maximal ideal $\fn \subset D$; so for any other maximal ideal $\fn'$, $I_{\fn'} = D_{\fn'}$. On the other hand, $A/fA \to D_{\fn} \stackrel{r}{\longrightarrow} A/fA$ is the identity, so since $A/fA \to D_{\fn}$ is $\fn D_{\fn}$-smooth, Lemma \ref{section} assures us that $I_{\fn}$ is generated by a regular sequence of length one --- that is, $I_{\fn} \cong D_{\fn}$. $D$, being finite over a local ring, is therefore semilocal, so the rank-one locally free module $I$ must be free. (Indeed, if we let $J$ be the Jacobson radical of $D$, then by the Chinese Remainder Theorem, $D/JD$ is a finite product of fields. It easily follows that $I/JI$ must be free over $D/JD$; whence $I$ is free by Nakayama's Lemma.)

One then obtains a short exact sequence of $D$-modules
\begin{equation}\tag{*}
0 \to D \to D \stackrel{r}{\longrightarrow} A/fA \to 0.
\end{equation}

Given a module $M \in \mcm^p(A/fA)$, consider the base-change functor 
\[ \mbf: \mcm^p(A/fA) \to \mcm^p(D) \mbox{ given by } M \mapsto M \otimes_{A/fA} D .\]
Since $A/fA \to D$ is flat, $\mbf$ is exact and preserves the codimensional filtration \cite[VII.5.2]{Quillen}. Since $A$ is finite over $D$ the forgetful functor $\mcm(D) \to \mcm(A)$ also preserves the codimensional filtration (Lemma \ref{finite_coniveau}) and so gives an exact functor $\mathbf{u}:\mcm^p(D) \to \mcm^p(A)$. Since tensoring over $A/fA$ will preserve the exactness of (*) we obtain a short exact sequence of exact functors $\mcm^{p}(A/fA) \to \mcm^p(A)$
\[ 0 \to \mathbf{u}\mathbf{F} \to \mathbf{u}\mathbf{F} \to \mathbf{i} \to 0 \]
where $\mathbf{i}$ is the natural inclusion. By \cite[III, Cor. 1]{Quillen}, $(\mathbf{u}\mathbf{F})_* = (\mathbf{u}\mathbf{F})_* + \mathbf{i}_*$ as homomorphisms $K_n(\mcm^p(A/fA)) \to K_n(\mcm^p(A))$ and the theorem follows.
\end{proof}

\begin{rem}Note that if $R$ is complete, then $A = R[[X_1, \cdots X_n]]$ is readily seen to satisfy the hypotheses of condition (b). In this way, Proposition \ref{main_thm_weak} may be seen as a generalization of the result of \cite{ReidSherman}.
\end{rem}

Before we address acyclicity results for the Gersten complex we shall need some basic lemmas about prime-avoidance.
\begin{lem}\label{avoidance}Let $(R,\fm_R,k) \to (A,\fm_A,L)$ be a flat, local map of Noetherian local rings such that $\fq := \fm_R A$ is prime.  Then the following statements hold:
\begin{enumerate}
\item[(i)] For $p \geq \dim R$, $\ds \mcm^{p+1}(A) = \varinjlim_{\mathclap{f \in \fm_A - \fq}}\mcm^p(A/fA)$.
\item[(ii)] Let $S \subset R$ be a multiplicative system so that $S \cap \fm_R \neq \emptyset$. Then for $p \geq \dim R - 1$, there is an equality $\ds \mcm^{p+1}(S^{-1}A) = \varinjlim_{\mathclap{f \in \fm_A - \fq}}\mcm^p(S^{-1}(A/fA))$.
\end{enumerate}
\end{lem}
\begin{proof}(i) Since $R \to A_{\fq}$ is flat, it follows that $\fq$ has height equal to $\dim R$. Thus, if $p \geq \dim R$, any $M \in \mcm^{p+1}(A)$ must vanish at $\fq$ and hence be annihilated by some $f \in \fm_A - \fq$. Since $f \in \fm_A-\fq$ is a non-zerodivisor on $A$ by Lemma \ref{localflatness1}, we have $M \in \mcm^p(A/fA)$.

(ii) Assume that $p \geq \dim R - 1$ and fix some $M \in \mcm^{p+1}(S^{-1}A)$. Let $\widetilde{\fp_1}, \cdots, \widetilde{\fp_m} \in \Spec S^{-1}A$ be the minimal primes of $M$ and let $\fp_i$ be their preimages in $A$. We claim that $\fp_i \not\subset \fq$. By assumption, $\dim A_{\fp_i} = \dim (S^{-1}A)_{\widetilde{\fp_i}} \geq p+1 \geq \dim R$. Thus, if it were the case that $\fp_i \subset \fq$, then we would have that $\fq = \fp_i$ since $\hgt \fq = \dim R$. This is impossible since $S \cap \fp_i = \emptyset$. Note that the support $\Supp M \subset \Spec S^{-1}A$ is the union of the vanishing sets of its minimal primes:
\[ V(\operatorname{ann(M)}) = \Supp M = V(\widetilde{\fp_1}) \cup V(\widetilde{\fp_2}) \cup \cdots \cup V(\widetilde{\fp_m}) = V(\widetilde{\fp_1} \widetilde{\fp_2} \cdots \widetilde{\fp_m}). \]
In other words, the product $\widetilde{\fp_1} \widetilde{\fp_2} \cdots \widetilde{\fp_m}$ and $\ann(M)$ agree up to radical, so for some $c >> 0$, $(\widetilde{\fp_1} \widetilde{\fp_2} \cdots \widetilde{\fp_m})^c \subset \ann(M)$. Choose $f_i \in \fp_i - \fq$ and put $f = (f_1 f_2 \cdots f_m)^c$. Its image in $S^{-1}A$ must then lie in the annihilator of $M$. Since $f \in \fm_A - \fq$, it must be a non-zerodivisor on $A$ and hence on $S^{-1}A$ as in part (i). Thus, $M \in \mcm^p(S^{-1}(A/fA))$ as desired.
\end{proof}

For any $A$ satisfying the hypotheses of Proposition \ref{main_thm_weak}, we are now poised to prove he following acyclicity result on its Gersten complex $\mcg_n^{\cdot}(A)$:

\begin{cor}\label{complex_vanish}Let $(R,\fm_R,k) \to (A,\fm_A,L)$ be a local morphism of Noetherian local rings. Suppose that either condition (a) or (b) of Proposition \ref{main_thm_weak} holds. Let $S \subset R$ be a multiplicative system so that $S \cap \fm_R \neq \emptyset$. Then the following are true:
\begin{enumerate}
\item[(i)] For all $p > \dim R$ and all $n \geq 0$, $H^p(\mcg_n^{\cdot}(A)) = 0$.
\item[(ii)] For all $p \geq \dim R$ and all $n \geq 0$, $H^p(\mcg_n^{\cdot}(S^{-1}A)) = 0$.
\end{enumerate}
\end{cor}
\begin{proof}For (i), a simple diagram chase shows (cf. \cite[VII.5.6]{Quillen}) shows that it suffices to prove that the natural inclusion $\mcm^{p+1}(A) \to \mcm^p(A)$ induces the zero map on $K$-theory for all $p \geq \dim R$. By assumption, $A$ satisfies the conditions of Lemma \ref{avoidance}; whence we have 
\[ \mcm^{p+1}(A) = \varinjlim_{\mathclap{f \in \fm_A - \fm_R A}}\mcm^p(A/fA) \mbox{ for all $p \geq \dim R$.} \]
On the other hand, Proposition \ref{main_thm_weak} shows that $K_n(\mcm^{p}(A/fA)) \to K_n(\mcm^p(A))$ is zero, and the claim is proved.

For part (ii), we need to show that $\mcm^{p+1}(S^{-1}A) \to \mcm^p(S^{-1}A)$ induces the zero map on $K$-theory for $p \geq \dim R - 1$. By part (ii) of Lemma \ref{avoidance}, we are reduced to showing that $K_n(\mcm^p(S^{-1}(A/fA))) \to K_n(\mcm^p(S^{-1}A))$ is zero for $f \in \fm_A - \fm_R A$. However, if we trace through the proof of Proposition \ref{main_thm_weak}, we notice that all rings in sight are $R$-algebras, so if we base-change via $- \otimes_R S^{-1}R$, everything will be preserved (e.g. $S^{-1}D \to S^{-1}(A/fA)$ will still have a principal kernel, etc.). We therefore obtain a proof of our claim.
\end{proof}

\section{Filtration of the Gersten Complex}\label{gersten_filter}
\subsection{One-dimensional base}\label{double_complex} Consider a DVR $(R, \pi R, k)$ and a flat map $\phi: X \to \Spec R$ having special and generic fibres $\overline{X}$ and $X[\frac{1}{\pi}]$ respectively. Gillet and Levine \cite[Cor. 7]{GilletLevine} show that the Gersten complexes $\mcg^{\cdot}_n(X[\frac{1}{\pi}])$ and $\mcg^{\cdot}_{n-1}(\overline{X})$ fit together into the following double complex:
\[ 
\xymatrixcolsep{3mm}
\xymatrix{
\ds 0 \ar[r] & \ds \bigoplus_{\mathclap{x \in X[\frac{1}{\pi}]^0}}K_n(k(x)) \ar[d] \ar[r] & \ds \bigoplus_{\mathclap{x \in X[\frac{1}{\pi}]^1}}K_{n-1}(k(x)) \ar[r] \ar[d] & \cdots \ar[r] & \ds \bigoplus_{\mathclap{x \in X[\frac{1}{\pi}]^{n-1}}}K_{1}(k(x)) \ar[d] \ar[r] & \ds \bigoplus_{\mathclap{x \in X[\frac{1}{\pi}]^n}}K_{0}(k(x)) \ar[d] \ar[r] & 0\\
\ds 0 \ar[r] & \ds \bigoplus_{\mathclap{x \in \overline{X}^0}}K_{n-1}(k(x)) \ar[r] & \ds \bigoplus_{\mathclap{{x \in \overline{X}^1}}}K_{n-2}(k(x)) \ar[r] & \cdots \ar[r] & \ds \bigoplus_{\mathclap{x \in \overline{X}^{n-1}}}K_{0}(k(x)) \ar[r] & 0 \ar[r] & 0\\
} \]
Since the vertical and horizontal differentials are induced by projecting those of $\mcg^{\cdot}_n(X)$ onto the appropriate summands, it is readily checked that each square anti-commutes and that totalizing the double complex indeed recovers $\mcg^{\cdot}_n(X)$. The advantage of this viewpoint is that we have reinterpreted the Gersten complex on $X$ in terms of the Gersten complexes on the equicharacteristic schemes $\overline{X}$ and $X[\frac{1}{\pi}]$. In particular, when $\phi$ is smooth, $\overline{X}$ and $X[\frac{1}{\pi}]$ are regular, thereby permitting us to recast the cohomology of the rows as the Zariski cohomology groups $H^p_{\Zar}(\overline{X}, \mck_{n-1})$ and $H^p_{\Zar}(X[\frac{1}{\pi}], \mck_{n})$ respectively (Lemma \ref{panin_sheaf}).

Fix a point $x \in X[\ovpi]^p \subset X^p$, and a point $y \in X^{p+1}$ with $y \in \overline{\left\{x\right\}}$. The map $K_n(k(x)) \to K_{n-1}(k(y))$ induced from $\mcg^{\cdot}_n(X)$ is realized in the above double complex by either a horizontal or vertical map, depending on whether $y \in X[\ovpi]^{p+1}$ or $y \in \overline{X}^p$. That we can decompose $\mcg_n^{\cdot}(X)$ into a double complex owes to the fact that $y$ can only live in one of two fibres. When the base has dimension two or greater, this is no longer true:

\begin{exmp}Let $S = \mathbb{C}[T_1,T_2]_{(T_1,T_2)}$, $Y = \Spec S$ and let $\phi: X = \Spec S[U,V] \to Y$ be the projection induced by $S \to S[U,V]$. The prime ideal $\fp = (T_1U + T_2) \in X^1$ contracts to $(0)$ in $S$, so $\phi(\fp) \in Y^0$. Consider now the primes $\fq_0 = (T_1U + T_2,V)$, $\fq_1 = (T_2,U)$, and $\fq_2 = (T_1,T_2)$ in $X^2$. For each $i$, $\fp \subset \fq_i$, so $\fq_i$ does indeed lie in the closure of $\fp$; however, it's easy to check that $\phi(\fq_i) \in Y^i$.
\end{exmp}

This shows that the double-complex approach is not viable in general; instead we shall use a spectral sequence.

\subsection{Construction of the Filtration}
Fix a flat morphism $f:X \to Y$ of Noetherian schemes. Our goal is to understand the Gersten complex on $X$ in terms of the fibres $X_y$ for each $y \in Y$. Put $Y^{(\geq p)} = \bigcup_{i \geq p}Y^i$. We define a decreasing filtration $F^{\bullet}\mcg_n^{\cdot}(X)$ via
\[ F^p\mcg_n^{q}(X) := \bigoplus_{\mathclap{\substack{x \in X^q \\ f(x) \in Y^{(\geq p)}}}}K_{n-q}(k(x)) \subset \bigoplus_{x \in X^q}K_{n-q}(k(x)) = \mcg_n^q(X). \]

\begin{lem}\label{filter1}The filtration $F$ is compatible with the differential on $\mcg_n^{\cdot}(X)$.
\end{lem}
\begin{proof}Write out $d: \mcg_n^{q}(X) \to \mcg_n^{q+1}(X)$ as
\[ \bigoplus_{x \in X^q} K_{n-q}(k(x)) \stackrel{d}{\longrightarrow} \bigoplus_{z \in X^{q+1}} K_{n-q-1}(k(z)), \]
and denote by $d_{xz}$ the induced map $K_{n-q}(k(x)) \to K_{n-q-1}(k(z))$. From the proof of \cite[VII.5.14]{Quillen}, if $d_{xz}$ is nonzero, then $z$ must lie in the closure of $\left\{x\right\}$. In other words, $x$ corresponds to a (non-maximal) prime of $\mco_{X,z}$; whence, $f(x)$ corresponds to a prime of $\mco_{Y,f(z)}$. We then have $\dim \mco_{Y,f(x)} \leq \dim \mco_{Y,f(z)}$ as desired.
\end{proof}

We may now take the associated-graded of this complex:
\[ F^p\mcg_n^{q}(X)/F^{p+1}\mcg_n^q(X) = \bigoplus_{\substack{x \in X^q \\ f(x) \in Y^p}}K_{n-q}(k(x)). \]

For a point $y \in Y$ with residue field $k(y)$ we shall denote by $X_y := X \times_{Y} \Spec k(y)$ the fibre over $y$. If $f(x) = y$ then since $f$ is flat, recall that there is an equality $\dim \mco_{X,x} = \dim \mco_{X_y,x} + \dim \mco_{Y,y}$ by \cite[15.1]{Matsumura}. Thus, if $x \in X^q$ and $f(x) \in Y^p$, then $x \in X_y^{q-p}$. This gives the following degree-wise identifications:
\[ F^p\mcg_n^{q}(X)/F^{p+1}\mcg_n^q(X) = \bigoplus_{y \in Y^p}\left( \bigoplus_{x \in X_y^{q-p}}K_{n-q}(k(x))\right) = \bigoplus_{y \in Y^p}\left( \mcg_{n-p}^{q-p}(X_y) \right) .\]

We claim that the induced differential on $F^p/F^{p+1}(\mcg^{\cdot}_n(X))$ acts diagonally on each fibre, allowing the above direct-sum decomposition of abelian groups to extend to a decomposition of complexes:

\begin{lem}\label{filter3}For each $p \geq 0$, there is a natural isomorphism of complexes
\[ (F^p/F^{p+1})\mcg_n^{\cdot}(X) \cong \bigoplus_{y \in Y^p}\left(\mcg_{n-p}^{\cdot}(X_y)[-p]\right).\]
\end{lem}
\begin{proof}We have already shown that the identification holds for each degree; we need to show that the differential is compatible. Let $y,y' \in Y^p$ with $y \neq y'$. Suppose that $x \in X_{y}^{q-p}$ and $x' \in X_{y'}^{q-p+1}$. Thus, $K_{n-q}(k(x))$ is a summand of $(F^p/F^{p+1})\mcg_n^{q}(X)$ and $K_{n-q-1}(k(x'))$ a summand of $(F^p/F^{p+1})\mcg_n^{q+1}(X)$. If $x'$ were in the closure of $\left\{x\right\}$, then as we saw in the proof of Lemma \ref{filter1}, $y = f(x)$ would correspond to a point of $\mco_{Y,y'}$. But since $y' \neq y$, we would then have $\dim \mco_{Y,y} < \dim \mco_{Y,y'}$, contradicting the fact that $y$ and $y'$ have the same codimension.  We therefore see that the induced differential (cf. Lemma \ref{filter1}) $d_{xx'}$ must be zero. This shows that the complex $(F^p/F^{p+1})\mcg_n^{\cdot}(X)$ decomposes as a direct sum.

It remains to show that for a fixed $y \in Y^p$, the Gersten complex on the fibre $\mcg_{n-p}^{\cdot}(X_y)[-p]$ is induced as a subquotient from the Gersten complex on $X$, $\mcg_n^{\cdot}(X)$. If $V \subset X$ is open, we have a map $\mcm^r(X) \to \mcm^r(V)$ and hence a map of coniveau spectral sequences and Gersten complexes (cf. Section \ref{coniveau}). This induced map $\mcg_n^{\cdot}(X) \to \mcg_n^{\cdot}(V)$ may be viewed as a quotient where we throw away all the summands involving points lying outside of $V$. On the other hand, if $T \to \Spec \mco_{Y,y}$ is flat, then $i:T_y = T \times \Spec k(y) \hookrightarrow T$ is a closed immersion of pure codimension $p$; in fact, by \cite[15.1]{Matsumura}, if $t$ is a point of codimension $r$ in $T_y$, then it has codimension $r+p$ in $T$. Thus, $i_*$ induces an inclusion $\mcm^r(T_y) \to \mcm^{r+p}(T)$ and hence a map of coniveau spectral sequences $E_1^{r,-n-r}(T_y) \to E_1^{r+p,-n-p-r}(T)$. Ultimately one obtains an inclusion of Gersten complexes $\mcg_{n-p}^{\cdot}(T_y)[-p] \to \mcg_n^{\cdot}(T)$ which can be thought of as the inclusion of all summands involving points contained in $T_y$. The result now follows from the fact that $X_y = T_y$ for $\ds T = \varprojlim_{U} (U \times_Y X)$ where $U$ varies over all open subsets of $Y$ containing $y$.
\end{proof}

\begin{prop}\label{ss_filtration}Let $X \to Y$ be a flat map of Noetherian schemes and assume that $\dim Y < \infty$. Then for each $n \geq 0$, there are bounded, first-quadrant spectral sequences:
\[ E_1^{pq} = \bigoplus_{y \in Y^p} H^q(\mcg_{n-p}^{\cdot}(X_y)) \Rightarrow H^{p+q}(\mcg_n^{\cdot}(X)). \]
\end{prop}
\begin{proof}Since $Y$ has finite Krull dimension, our filtration $F^{\bullet}$ on $\mcg_n^{\cdot}(X)$ is bounded, guaranteeing the convergence of the associated spectral sequence (cf. \cite[5.5.1]{Weibel})
\[ E_1^{pq} = H^{p+q}(F^p/F^{p+1} \mcg_n^{\cdot}(X)) \Rightarrow H^{p+q}(\mcg_n^{\cdot}(X)). \]
On the other hand, by Lemma \ref{filter3} we can identify the $E_1$ page as a direct sum of complexes:
\[ \begin{array}{r@{\hspace{2pt}}c@{\hspace{2pt}}l@{\hspace{2pt}}}
\ds E_1^{pq} = H^{p+q}(F^p/F^{p+1} \mcg_n^{\cdot}(X)) & = & \ds H^{p+q}\left(\bigoplus_{y \in Y^p}\left(\mcg_{n-p}^{\cdot}(X_y)[-p]\right)\right) \vspace{3mm}\\
																								  & = &\ds \bigoplus_{y \in Y^p} H^q(\mcg_{n-p}^{\cdot}(X_y)). \end{array}\]
\end{proof}

\subsection{Acyclicity of the Gersten Complex for Geometrically Regular Maps} Now that we can understand the Gersten complex in terms of its fibres, we can apply Panin's trick to prove our main theorem:
\begin{thm}\label{mainthm}Let $(R,\fm_R,k) \to (A,\fm_A,L)$ be a geometrically regular morphism of Noetherian local rings. Suppose $A$ is Henselian. Then $H^m(\mcg_n^{\cdot}(A)) = 0$ for all $m > \dim R$ and all $n \geq 0$.
\end{thm}
\begin{proof}By Popescu's Theorem \ref{popescu}, we may write $A = \fcolim A_{\lambda}$ where $\left\{A_{\lambda}\right\}_{\lambda \in \Lambda}$ is a filtered system of smooth $R$-algebras. If we localize each $A_{\lambda}$ at the inverse image of $\fm_A$, we may assume each $A_{\lambda}$ is local and essentially smooth over $R$. By the universal property of Henselization, each $A_{\lambda} \to A$ factors uniquely through $A_{\lambda}^h$, so after replacing $A_{\lambda}$ with $A_{\lambda}^h$, we can assume that $A$ is a filtered colimit of $R$-algebras $\left\{A_{\lambda}\right\}_{\lambda \in \Lambda}$ satisfying condition (a) of Proposition \ref{main_thm_weak}.

Put $Y = \Spec R$, $X = \Spec A$, and $X_{\lambda} = \Spec A_{\lambda}$.  For $y \in Y$, we put $X_{y} := X \times_{Y} \Spec k(y)$ and $(X_{\lambda})_{y} := X_{\lambda} \times_Y \Spec k(y)$. Clearly, we have $X_{y} = \finvlim (X_{\lambda})_{y}$. Since $X_{y}$ is regular and equicharacteristic, Lemmas \ref{panin_sheaf} and \ref{panin_limit} give
\[ H^m(\mcg_n^{\cdot}(X_{y})) = H^m_{\Zar}(X_{y},\mck_n) = \fcolim H^m_{\Zar}((X_{\lambda})_{y},\mck_n) \mbox{ for all $m,n \geq 0$.} \]

Fix $\lambda \in \Lambda$. We study now the Gersten complex on each $(X_{\lambda})_{y}$. If $y \in Y^p$, then $y$ corresponds to a height-$p$ prime, $\fp$. One then trivially has that $\dim R/\fp \leq \dim R - p$. The local homomorphism $R/\fp \to A_{\lambda}/\fp A_{\lambda}$ also satisfies condition (a) of Proposition \ref{main_thm_weak}. $(X_{\lambda})_y = \Spec(A_{\lambda} \otimes_R k(y))$, and $A_{\lambda} \otimes_R k(y) \cong S^{-1}(A_{\lambda}/\fp A_{\lambda})$ where $S = R/\fp - \left\{0\right\}$. If $p < \dim R$ (i.e. $\fp \neq \fm_R$), then $S$ will have non-empty intersection with the maximal ideal of $R/\fp$. In this case, we can apply part (ii) of Corollary \ref{complex_vanish} and Lemma \ref{panin_sheaf} to conclude that
\[ H^m_{\Zar}((X_{\lambda})_{y},\mck_n) = H^m(\mcg^{\cdot}_n((X_{\lambda})_{y})) = 0 \mbox{ if $p < \dim R$, $\dim R - p \leq m$, and $n \geq 0$.} \]

On the other hand, if $p = \dim R$, then $\fp$ is the maximal ideal of $R$ and $X_{y}$ is an equicharacteristic regular local scheme. Thus, $H^m(\mcg_n^{\cdot}(X_{y})) = 0$ for $m \geq 1$ by Panin's Theorem \ref{panin}. Putting everything together therefore shows that
\[ H^m(\mcg_n^{\cdot}(X_{y})) = H^m_{\Zar}(X_{y},\mck_n) = \fcolim H^m_{\Zar}((X_{\lambda})_y, \mck_n) = 0 \]
for all $y \in Y^p$, $m \geq \max \left\{\dim R - p,1 \right\}$, and $n \geq 0$.

Proposition \ref{ss_filtration} gives us a first-quadrant spectral sequence
\[ E_1^{pq} = \bigoplus_{y \in Y^p} H^q(\mcg_{n-p}^{\cdot}(X_{y})) \Rightarrow H^{p+q}(\mcg_n^{\cdot}(X)). \]
where we see that $E_1^{pq} = 0$ for $q \geq \max \left\{\dim R - p,1\right\}$ (or, of course, if $p > \dim R$). For $m > \dim R$, $H^m(\mcg_n(X))$ acquires a finite filtration whose associated graded is $\ds \bigoplus_{p \geq 0} E_{\infty}^{p,m-p} = 0$. The theorem follows.
\end{proof}

\begin{quest}If $\dim R = 1$ in Theorem \ref{mainthm}, then the $E_1$ page of the spectral sequence \ref{ss_filtration} will consist of a single row along $q=0$.
\[ 0 \to E_1^{0,0} = \bigoplus_{\mathclap{\substack{\fp \subset R \\ \hgt \fp = 0}}}K_n(A \otimes_R k(\fp)) \to E_1^{1,0} = \bigoplus_{\mathclap{\substack{\fp \subset R \\ \hgt \fp = 1}}}K_{n-1}(k(\fp)) \to 0 \]
Even if $R$ is a local domain of dimension $2$ with fraction field $F$, this is no longer necessarily true. In fact, our methods cannot determine whether $H^1(\mcg_n^{\cdot}(A \otimes_R F)) = E_1^{0,1}$ vanishes. To find a counterexample, it would suffice to find a Henselian $A$ geometrically regular over $R$ such that $A \otimes_R F$ is not a UFD. In such a case, $H^1(\mcg_1(A \otimes_R F)) = \ch^1(A \otimes_R F) \neq 0$ (see Section \ref{chow} below). Note that if such an example exists, then $R$ would have to be singular.
\end{quest}

Theorem \ref{mainthm} may be regarded as a Nisnevich-local acyclicity statement in a sense which we now make precise:
\begin{cor}\label{nisnevich_acyclicity}Let $X \to Y$ be a geometrically regular morphism of Noetherian schemes with $d = \dim Y < \infty$. Put
\[ \tau^{\leq d}\left(\dumcg_{n,X}^{\cdot}\right) = \left(0 \to \dumcg_{n,X}^0 \to \dumcg_{n,X}^1 \to \cdots \to \dumcg^{d-1}_{n,X} \to \ker\left( \dumcg^{d}_{n,X} \to \dumcg^{d+1}_{n,X} \right) \to 0 \right) \]
where $\dumcg_{n,X}$ is the Gersten complex of Nisnevich sheaves from Section \ref{nisnevich}. Then the natural inclusion $\ds \tau^{\leq d}\left(\dumcg_{n,X}^{\cdot}\right) \to \dumcg_{n,X}^{\cdot}$ is a quasi-isomorphism on $X_{\Nis}$ for all $n \geq 0$.
\end{cor}
\begin{proof}Fix some $x \in X$. It suffices to show that if we take the Nisnevich stalk of the complex $\dumcg_{n,X}^{\cdot}$ at $x$ becomes acyclic in degrees $d+1$ and higher. By Lemma \ref{complex_stalk}, we need to show that $H^p(\mcg_n^{\cdot}(\mco_{X,x}^h)) = 0$ for $p > d$. If $x$ maps to $y \in Y$, then $\mco_{Y,y} \to \mco_{X,x} \to \mco_{X,x}^h$ is geometrically regular, $\dim \mco_{Y,y} \leq d$, and the result now follows from Theorem \ref{mainthm}.
\end{proof}

\begin{rem}To prove analogous statement of Corollary \ref{nisnevich_acyclicity} for $X_{\Zar}$, one would need to remove the Henselian hypothesis from Theorem \ref{mainthm}. To do this, it would suffice to prove Proposition \ref{main_thm_weak} when $R \to A$ is an essentially smooth local morphism of Noetherian local rings.
\end{rem}

\section{The Case of a One-Dimensional Base}\label{unramified} When $(A,\fm_A, L)$ is essentially smooth over a DVR $(R, \pi R, k)$ \cite{GilletLevine} and \cite{Bloch} prove a Zariski-local analogue of Theorem \ref{mainthm}. Using Panin's machinery, we can remove the finiteness hypotheses and generalize their results to the geometrically regular case.

\begin{prop}\label{gl_improved}Let $(A,\fm_A,L)$ be a Noetherian local ring, geometrically regular over a DVR $(R, \pi R, k)$. Then there is a quasi-isomorphism of complexes:
\[ \textstyle \left( 0 \to  K_n(A[\ovpi]) \stackrel{\partial}{\longrightarrow} K_{n-1}(A/\pi A) \to 0 \right) \stackrel{\sim}{\longrightarrow} \mcg_n^{\cdot}(A). \]
(Here, $\partial$ is the connecting map from the localization exact sequence for $A/\pi A$, $A$ and $A[\ovpi]$.)
\end{prop}
\begin{proof}By Popescu's Theorem \ref{popescu} we can write $A$ as a filtered colimit of local rings $A_{\lambda}$, each of which is essentially smooth over $R$. Put $X = \Spec A$ and $X_{\lambda} = \Spec A_{\lambda}$. Let $X[\ovpi] = \Spec A[\ovpi]$, $\overline{X} = \Spec A/\pi A$, and define $X_{\lambda}[\ovpi]$ and $\overline{X}_{\lambda}$ similarly.

As in the proof of \cite[Cor. 7]{GilletLevine}, the double-complex from subsection \ref{double_complex} receives a map as indicated below:
\begin{equation}\tag{*}
\xymatrixcolsep{5mm}
\begin{array}{l@{\hspace{1pt}}l@{\hspace{1pt}}l@{\hspace{1pt}}}
\left(\vcenter{\xymatrix{
K_n(X[\ovpi]) \ar_{\partial}[d] \\
K_{n-1}(\overline{X})
}}\right)
& 
\longrightarrow
& 
\left( \vcenter{\xymatrix{
\ds \bigoplus_{\mathclap{x \in X[\ovpi]^0}}K_n(k(x)) \to \cdots \ar[r] \ar[d] & \ds \bigoplus_{\mathclap{x \in X[\ovpi]^{n-1}}}K_1(k(x)) \ar[r] \ar[d] & \ds \bigoplus_{\mathclap{x \in X[\ovpi]^n}}K_0(k(x)) \ar[d] \\
\ds \bigoplus_{\mathclap{x \in \overline{X}^0}}K_{n-1}(k(x)) \to \cdots \ar[r] & \ds \bigoplus_{\mathclap{x \in \overline{X}^{n-1}}}K_{0}(k(x)) \ar[r] & 0
}}\right)
\end{array}
\end{equation}

Our claim is that the induced map on the totalizations is a quasi-isomorphism. On the right-hand side of (*), the top row is simply the Gersten complex $\mcg_n^{\cdot}(X[\ovpi])$ while the bottom row is $\mcg_{n-1}(\overline{X})$. It will suffice to show that these complexes resolve $K_n(X[\ovpi])$ and $K_{n-1}(\overline{X})$ respectively. Since $\overline{X}$ is an equicharacteristic, regular local scheme, Panin's Theorem \ref{panin} produces the desired quasi-isomorphism $K_{n-1}(\overline{X}) \stackrel{\sim}{\longrightarrow} \mcg_{n-1}(\overline{X})$. 

We are left to proving the analogous statement for $X[\ovpi]$. Here we use the fact that $X[\ovpi] = \ds \varprojlim_{\lambda \in \Lambda}X_\lambda$. Since each $X_{\lambda}$ is essentially smooth over the DVR $R$, it follows that the generic fibre $X_{\lambda}[\ovpi]$ satisfies Gersten's Conjecture by \cite[4(c)]{GilletLevine} or \cite{Bloch}. As they are regular schemes defined over a field, we have by Corollary \ref{panin_sheaf} that
\[ \textstyle H^q_{\Zar}(X_{\lambda}[\ovpi], \mck_n) = H^q(\mcg_n^{\cdot}(X_{\lambda}[\ovpi])) = \left\{ \begin{array}{ll} K_n(X_{\lambda})[\ovpi] & q = 0 \\ 0 & q > 0 \end{array}.\right. \]
On the other hand, $X[\ovpi]$ is likewise regular and equicharacteristic, so again by Corollary \ref{panin_sheaf}, we have that $H^q_{\Zar}(X[\ovpi],\mck_n) = H^q(\mcg_n^{\cdot}(X[\ovpi]))$. Invoking Corollary \ref{panin_limit} shows that 
\[ \textstyle H^q_{\Zar}(X[\ovpi],\mck_n) = \ds \varinjlim_{\lambda \in \Lambda} \textstyle H^q_{\Zar}(X_{\lambda}[\ovpi], \mck_n) = 0 \mbox{ for } q > 0. \]

When $q = 0$, the isomorphism $K_n(X[\ovpi]) \stackrel{\sim}{\longrightarrow} H^0(X[\ovpi],\mck_n)$ follows from both Corollary \ref{panin_limit} and the isomorphism $\ds \varinjlim_{\lambda \in \Lambda} \textstyle K_n(X_{\lambda}[\ovpi]) \stackrel{\sim}{\longrightarrow} K_n(X[\ovpi])$ from \cite[VII.2.2]{Quillen} (see also the proof of \cite[Lem. 1.3]{Panin}).
\end{proof}

Proposition \ref{gl_improved} has an immediate global corollary.
\begin{cor}\label{unramified_bloch}(cf. \cite[Cor. 8]{GilletLevine}) Let $Y$ be an integral, regular Noetherian scheme of dimension $1$ with generic point $\eta$. Suppose $X \to Y$ is geometrically regular. Then there is a quasi-isomorphism of complexes on $X_{\Zar}$ for all $n \geq 1$:
\[ \left( 0 \to (u_{\eta})_*(\mck_{n,X_{\eta}}) \stackrel{\partial}{\longrightarrow} \bigoplus_{y \in Y^1}(u_y)_*(\mck_{n-1, X_y}) \to 0 \right) \stackrel{\sim}{\longrightarrow} \dumcg_{n,X}^{\cdot} \]
(Here, $u_t$ is the natural inclusion of the fibre $X_t \to X$.)
\end{cor}
\begin{proof}We return to the sheaf-theoretic Gersten complexes defined in Section \ref{nisnevich}. For any $t \in Y$ and any open $U \subset X$, the sections of $(u_t)_*\left( \dumcg^{\cdot}_{n,X_t} \right)$ over $U$ simply recover the Gersten complex $\mcg_n(U_t)$. Furthermore, just like the double complex introduced in Section \ref{double_complex}, $\mcg_n(U)$ is seen to be the totalization of the double complex
\[ 
\xymatrixcolsep{3.5mm}
\xymatrix{
\ds 0 \ar[r] & \ds \bigoplus_{\mathclap{x \in U_{\eta}^0}}K_n(k(x)) \ar[d] \ar[r] & \ds \bigoplus_{\mathclap{x \in U_{\eta}^1}}K_{n-1}(k(x)) \ar[r] \ar[d] & \cdots \ar[r] & \ds \bigoplus_{\mathclap{x \in U_{\eta}^{n-1}}}K_{1}(k(x)) \ar[d] \ar[r] & \ds \bigoplus_{\mathclap{x \in U_{\eta}^n}}K_{0}(x) \ar[d] \ar[r] & 0\\
\ds 0 \ar[r] & \ds \bigoplus_{\mathclap{\substack{x \in U_y^0 \\ y \in Y^1}}}K_{n-1}(k(x)) \ar[r] & \ds \bigoplus_{\mathclap{\substack{x \in U_y^1 \\ y \in Y^1}}}K_{n-2}(k(x)) \ar[r] & \cdots \ar[r] & \ds \bigoplus_{\mathclap{\substack{x \in U_y^{n-1} \\ y \in Y^1}}}K_{0}(k(x)) \ar[r] & 0 \ar[r] & 0\\
} \]
The top row is just $\mcg_n^{\cdot}(U_{\eta})$, and by the argument in Lemma \ref{filter3}, the bottom row decomposes into $\ds \bigoplus_{y \in Y^1}\mcg_{n-1}^{\cdot}(U_y)$. Thus, we see that $\dumcg^{\cdot}_{n,X}$ arises as the totalization of a double complex of sheaves on $X_{\Zar}$ which we shall denote by
\[ 0 \to (u_{\eta})_*\left(\dumcg_{n,X_{\eta}}^{\cdot}\right) \to \bigoplus_{y \in Y^1}(u_y)_* \left( \dumcg_{n-1,X_y}^{\cdot} \right) \to 0. \]
As in Proposition \ref{gl_improved}, we have a map of double complexes:
\[
\begin{array}{l@{\hspace{1pt}}l@{\hspace{1pt}}l@{\hspace{1pt}}}
\left(\vcenter{\xymatrix{
(u_{\eta})_*\left(\mck_{n,X_{\eta}}\right) \ar_{\partial}[d] \\
\ds \bigoplus_{y \in Y^1}(u_y)_*\left(\mck_{n-1,X_y}\right)
}}\right)
& 
\longrightarrow
& 
\left( \vcenter{\xymatrix{
(u_{\eta})_*\left(\dumcg_{n,X_{\eta}}^{\cdot}\right) \ar[d] \\
\ds \bigoplus_{y \in Y^1}(u_y)_* \left( \dumcg_{n-1,X_y}^{\cdot} \right)
}}\right)
\end{array}
\]
This map induces a quasi-isomorphism of total complexes: If we fix $x \in X_y$ for some $y \in Y^1$, then $\mco_{X,x}$ is geometrically regular over the DVR $\mco_{Y,y}$; so taking the Zariski stalk of the above just recovers the diagram (*) from Proposition \ref{gl_improved}. If $x \in X_{\eta}$, then $\mco_{X,x}$ is an equicharacteristic regular local ring. Taking the stalk at $x$ in this case will cause the bottom row to vanish and the top to become $K_n(\mco_{X,x}) \stackrel{\sim}{\rightarrow} \mcg_n(\mco_{X,x})$.
\end{proof}

\subsection{Gersten's Conjecture for Unramified Regular Local Rings}Recall that a mixed characteristic regular local ring $(A,\fm_A,L)$ is unramified if $p \in \fm_A - \fm_A^2$ for some prime number $p$. This is enough to guarantee that the map $\mathbb{Z}_{(p)} \to A$ is geometrically regular (see Example \ref{geomregexmps}). Our goal in this section is to reduce the general Gersten Conjecture for such rings $A$ to the following, ostensibly simpler one:

\begin{conj}[Gersten's Arithmetic DVR Conjecture]\label{gdvr} Let $(R,pR,k)$ be a DVR, essentially smooth over $\mathbb{Z}$. Then the transfer map $K_n(k) \to K_n(R)$ is zero for all $n \geq 0$.
\end{conj}

When $R$ is a DVR, the equivalence between Conjecture \ref{gdvr} and the Gersten Conjecture for $R$ follows immediately from Quillen's localization theorem. As the name suggests, Conjecture \ref{gdvr} is amenable to attack via the methods of arithmetic geometry as we shall demonstrate after proving the following:

\begin{thm}\label{unramifiedconj}Gersten's Arithmetic DVR Conjecture implies the full Gersten Conjecture for unramified regular local rings $(A,\fm_A, L)$.
\end{thm}
\begin{proof}We shall assume that $A$ has mixed characteristic $p > 0$. By Proposition \ref{gl_improved}, we have the quasi-isomorphism 
\[ \textstyle \left( 0 \to  K_n(A[\ovp]) \stackrel{\partial}{\longrightarrow} K_{n-1}(A/pA) \to 0 \right) \stackrel{\sim}{\longrightarrow} \mcg_n^{\cdot}(A), \]
so we need only show that $\partial$ is surjective with kernel equal to $K_n(A)$. From the localization sequence
\[\textstyle \cdots \to K_n(A/pA) \to K_n(A) \to K_n(A[\ovp]) \stackrel{\partial}{\longrightarrow} K_{n-1}(A/pA) \to K_{n-1}(A) \to  \cdots \]
we see that proving Gersten's Conjecture for $A$ amounts to proving that the transfer map $t: K_n(A/pA) \to K_n(A)$ is zero for all $n \geq 0$. 

Once again by Popescu's Theorem \ref{popescu}, we can write $\ds A = \fcolim A_{\lambda}$ with each $A_{\lambda}$ a local, essentially smooth $\mathbb{Z}_{(p)}$-algebra. For a fixed $\lambda$, we consider the pushout diagram:
\[
\xymatrix{
A_\lambda \ar[r] \ar[d] & A_{\lambda}/pA_{\lambda} \ar[d] \\
A \ar[r] & A/pA
}
\]
Note that the vertical maps have finite $\Tor$-dimension by virtue of all rings being regular; the horizontal maps correspond to closed immersions of the corresponding schemes. Since $p$ is a regular parameter on $A_{\lambda}$, we may identify $\Tor_q^{A_{\lambda}}(A,A_{\lambda}/pA_{\lambda})$ with the $q$-th Koszul homology group $H_q(p,A)$. As $p$ is likewise a regular parameter on $A$, these must vanish for $q > 0$, and we may conclude that $A$ and $A_{\lambda}/p A_{\lambda}$ are $\Tor$-independent over $A_{\lambda}$. Using \cite[VII.2.11]{Quillen}, we have a commutative diagram
\[
\xymatrix{
K_n(A_\lambda/pA_{\lambda}) \ar^{t}[r] \ar[d] & K_n(A_{\lambda}) \ar[d] \\
K_n(A/pA) \ar^{t}[r] & K_n(A)
}
\]
(Here, we are using the regularity of our rings to identify $K$-theory with $K'$-theory.)

Since $A/pA = \fcolim A_{\lambda}/p A_{\lambda}$, every class $\alpha \in K_n(A/pA)$ lies in the image of $K_n(A_{\lambda'}/p A_{\lambda'}) \to K_n(A/pA)$ for some $\lambda'$. It therefore suffices to show that every transfer map $K_n(A_{\lambda}/p A_{\lambda}) \to K_n(A_{\lambda})$ is zero.

By \cite[Cor. 6]{GilletLevine}, knowing Gersten's Conjecture for the DVR $(A_{\lambda})_{(p)}$ implies the conjecture for $A_{\lambda}$, itself. By construction, each of these DVRs is essentially of finite-type over $\mathbb{Z}$, so if we assume Conjecture \ref{gdvr}, then each $A_{\lambda}$ will indeed satisfy Gersten's Conjecture. The transfer map $t: K_n(A_{\lambda}/pA_{\lambda}) \to K_n(A_{\lambda}) \cong K_n(\mcm^0(A_{\lambda}))$ factors through $K_n(\mcm^1(A_\lambda))$ and therefore must vanish by \cite[VII.5.6]{Quillen}. 
\end{proof}

\subsection{Finite Coefficients}We very briefly review basic properties of $K$-theory with finite coefficients. For details, see \cite[IV.2]{Kbook}. If $\mce$ is an exact category, we denote by $K(\mce)$ the space whose $n$th homotopy group is $K_n(\mce)$ (to be concrete, take $K(\mce) = \Omega BQ\mce$ \cite[II]{Quillen}). For $\ell \in \mathbb{Z}$, define the Moore space $P^m(\mathbb{Z}/\ell \mathbb{Z})$ to be the cofibre of the unique (up to homotopy) degree-$\ell$ map on the sphere $S^{m-1} \stackrel{\ell}{\longrightarrow} S^{m-1}$. We then define $K_n(\mce;\modl) := [P^n(\modl),K(\mce)]$, where the brackets denote homotopy classes of basepoint-preserving maps. Since our model of $K(\mce)$ is an infinite loop space, $K_n(\mce;\modl)$ is an abelian group for all $n \geq 0$.

By the contravariance of $[-,K(\mce)]$, the map $S^n \to P^{n+1}(\modl)$ induces a natural map $\Psi: K_{n+1}(\mce;\modl) \to K_n(\mce)$; from the universal coefficients theorem, the image of $\Psi$ is known precisely:
\begin{lem}\label{fincoeff}The image of the map $\Psi: K_{n+1}(\mce;\modl) \to K_n(\mce)$ is precisely the subgroup $\Tor_1^{\mathbb{Z}}(K_n(\mce),\modl) = \left\{\alpha \in K_n(\mce): \ell \cdot \alpha = 0 \right\}$.
\end{lem}

Thanks to Gillet \cite{GilletDVR} and Geisser-Levine \cite[8.2]{GeisserLevine}, Gersten's Conjecture is known for all DVRs if one passes to finite coefficients:
\begin{thm}\label{fcdvr} Let $(R,\pi R,k)$ be a DVR. Then for all $\ell \in \mathbb{Z}$ and all $n \geq 0$, the transfer map $K_n(k;\modl) \to K_n(R;\modl)$ is zero.
\end{thm}

It therefore follows from Theorems \ref{unramifiedconj} and \ref{fcdvr} that the full Gersten Conjecture for an arbitrary unramified regular ring is known in the case of finite coefficients:
\begin{cor}If $(A,\fm_A,L)$ is any unramified regular local ring, then for $\ell \in \mathbb{Z}$,
\begin{multline*} 0 \to K_n(A;\modl) \to \bigoplus_{\mathclap{\hgt \fp = 0}}K_n(k(\fp);\modl) \to \bigoplus_{\mathclap{\hgt \fp = 1}}K_{n-1}(k(\fp);\modl) \to \cdots  \\
\cdots \to \bigoplus_{\mathclap{\hgt \fp = n-1}}K_1(k(\fp);\modl) \to \bigoplus_{\mathclap{\hgt \fp = n}}K_0(k(\fp);\modl) \to 0
\end{multline*}
is exact.
\end{cor}

More importantly, it is possible in some cases to extend the result of Theorem \ref{fcdvr} to integer coefficients:
\begin{cor}\label{dvrintegral}Let $(R,\pi R, k)$ be a DVR. Fix $n \geq 0$ and suppose that $K_n(k)$ is a torsion group. Then the transfer map $K_n(k) \to K_n(R)$ is zero.
\end{cor}
\begin{proof}Fix $\alpha \in K_n(k)$ and let $\ell \in \mathbb{Z}$ be such that $\ell \cdot \alpha = 0$. Since the transfer map is induced from the inclusion $\mcm^1(R) \to \mcm^0(R)$, it is compatible with the map $\Psi$ defined earlier in this section and so produces a commutative diagram:
\[ \xymatrix{
K_{n+1}(k;\modl) \ar^{\Psi}[d] \ar^0[r] & K_{n+1}(R;\modl) \ar^{\Psi'}[d] \\
K_n(k) \ar[r] & K_n(R)
}\]
By Lemma \ref{fincoeff}, $\alpha$ is contained in the image of $\Psi$, and by Theorem \ref{fcdvr}, the top map is zero. Thus, $\alpha$ maps to zero as claimed.
\end{proof}

\subsection{Torsion in the $K$-theory of Fields}If $(R,\pi R, k)$ is a mixed characteristic DVR essentially of finite-type over $\mathbb{Z}$, then by the Nullstellensatz, $k = R/\pi R$ will be a finitely-generated field over $\mathbb{F}_p$ -- that is, the function field of some (possibly singular) projective variety. In view of Corollary \ref{dvrintegral}, understanding the torsion in $K_n(k)$ is a viable approach to attacking Conjecture \ref{unramifiedconj}.

\begin{lem}\label{parshin_special}Let $X$ be an integral scheme of finite-type over $\mathbb{F}_p$ with generic point $\eta$ and function field $L = k(\eta)$. Suppose that for all $x \neq \eta$ in $X$, $K_n(k(x))$ is torsion for $n > \trdeg(k(x)/\mathbb{F}_p)$. Then there are isomorphisms $K'_n(X)_{\mathbb{Q}} \cong K_n(L)_{\mathbb{Q}}$ for all $n > \dim X$.
\end{lem}
\begin{proof}By the coniveau spectral sequence (Theorem \ref{BGQ}), we have
\[ E_1^{pg} = \bigoplus_{\mathclap{x \in X^p}}K_{-p-q}(k(x)) \Rightarrow K'_{-p-q}(X). \]
Since $X$ is integral, $K_q(L) = E_1^{0,-q}$. It therefore suffices to show that if we apply $- \otimes_{\mathbb{Z}} \mathbb{Q}$ to the spectral sequence, then $(E_1^{pq})_{\mathbb{Q}} = 0$ for $q < -\dim X$ and $p > 0$. It follows from \cite[15.6]{Matsumura} that if $y \in X^p$ then $y$ is the generic point of a closed subvariety having dimension $\dim X - p$; whence $\trdeg(k(y)/\mathbb{F}_p) = \dim X - p$. By hypothesis, it follows that when $p>0$, $(E_1^{pq})_{\mathbb{Q}} = 0$ whenever $-p-q > \dim X - p$; that is, when $q < - \dim X$.
\end{proof}

When the conditions of Lemma \ref{parshin_special} are met, we see that save for a few degrees, the rational $K$-theory of a function field over $\mathbb{F}_p$ coincides with that of any variety that models it. A particularly relevant conjecture is the following:

\begin{conj}[Beilinson-Parshin Conjecture] \cite[51]{Kahn} Let $X$ be a smooth, projective variety over $\mathbb{F}_p$. Then $K_n(X)$ is torsion for all $n > 0$.
\end{conj}

Using an argument similar to that of Lemma \ref{parshin_special}, Geisser has shown \cite[1.3]{Geisser} that the Beilinson-Parshin Conjecture would imply that $K_n(L)_{\mathbb{Q}} = 0$ whenever $\trdeg(L/\mathbb{F}_p) < n$. It is, in fact, impossible to do any better as the following Proposition demonstrates:

\begin{prop}For all fields $L$ of characteristic $p > 0$, if $\trdeg(L/\mathbb{F}_p) \geq n$, then $K_i(L)_{\mathbb{Q}} \neq 0$ for all $i \leq n$.
\end{prop}
\begin{proof}We argue by induction on $n$ with the case of $n = 0$ being immediate. Suppose $n > 0$ and that $\trdeg(L/\mathbb{F}_p) \geq n$. Fix a transcendence basis $\left\{x_{\lambda}\right\}_{\lambda \in \Lambda}$ for $L/\mathbb{F}_p$ and choose some $x_{\lambda'}$ in this set. Let $E$ be a purely transcendental extension of $\mathbb{F}_p$ on the set $\left\{x_{\lambda}\right\}$ for $\lambda \in \Lambda - \left\{\lambda'\right\}$. Thus, $\trdeg(E/\mathbb{F}_p) \geq n-1$.  If we define $A := E[X] \to L$ via $X \mapsto x_{\lambda'}$, then $L$ is algebraic over the fraction field of $F$ of $A$.

Observe that it suffices to prove that $K_i(F)_{\mathbb{Q}} \neq 0$ for $i \leq n$. Indeed, for any finite extension $F'$ of $F$, the composite $K_i(F) \to K_i(F') \to K_i(F)$ defined by base-extension and transfer is simply multiplication by $d = [F':F]$. Thus, $K_i(F)_{\mathbb{Q}} \to K_i(F')_{\mathbb{Q}}$ is injective. It follows then that $K_i(F)_{\mathbb{Q}} \to K_i(L)_{\mathbb{Q}}$ must also injective as $L/F$ is algebraic and hence a direct limit of finite extensions.

Since $A = E[X]$ is a UFD, we can realize its fraction field $F$ as a direct limit of rings $A_f$ as $f$ varies over square-free elements of $A$. (Note that this system is directed: if $f$ and $g$ are square-free then both $A_f$ and $A_g$ map to $A_h = A_{fg}$ where $h = \operatorname{lcm}(f,g)$.) We now make the following intermediate claim: \vspace{2mm}

\noindent\textbf{Subclaim: }\emph{If $f,g \in A$ are square-free (or units), then for all $m \geq 0$, the transfer map $K_m'(A_f/gA_f) \to K_m'(A_f)$ is zero.} \vspace{2mm} \newline
To prove this claim, we mimic the proof of Proposition \ref{main_thm_weak}. Set $D = (A_f/gA_f) \otimes_E A_f$. As $g$ is square-free, $A/gA$ is just a product of fields, each a finite-extension of $E$; so too is its localization $A_f/g A_f$. Thus, $E \to A_f/g A_f$ is finite, thereby making $A_f \to D$ finite. Furthermore, $A_f/g A_f \to D$ is smooth because $E \to A_f$ is.

We have the retraction $r:D \to A_f/g A_f$ given by $a \otimes b \mapsto ab$ which makes $A_f/g A_f \to D \stackrel{r}{\longrightarrow} A_f/g A_f$ the identity. By Lemma \ref{section}, $I = \ker r$ is a locally-free $D$-module of rank $1$. We now show that $I \cong D$. As mentioned above, $A_f/g A_f \cong \prod_{j=1}^{k}E_j'$ where each $E'_j$ is a finite field extension of $E$. $D = A_f/g A_f \otimes_E A_f$ is just a localization of $A_f/g A_f \otimes_E E[X] \cong \prod_{j=1}^{k}E_j'[X]$ and is therefore a principal ideal ring. Thus, $I$ must be principal (and consequently isomorphic to $D$ as $I$ is projective). We therefore get a short exact sequence:
\[ 0 \to D \to D \stackrel{r}{\longrightarrow} A_f/g A_f \to 0. \]
The functor $\mathbf{F}: \mcm(A_f/g A_f) \to \mcm(D)$ given by $M \mapsto M \otimes_{A_f/g A_f} D$ is exact as are the forgetful functors $\mathbf{u}: \mcm(D) \to \mcm(A_f)$ and $\mathbf{i}: \mcm(A_f/g A_f) \to \mcm(A_f)$. From the exact sequence, we obtain an exact sequence of functors $\mcm(A_f/g A_f) \to \mcm(A_f)$:
\[0 \to \mathbf{u}\mathbf{F} \to \mathbf{u}\mathbf{F} \to \mathbf{i} \to 0 \]
which proves that $\mathbf{i}_*:K_m'(A_f/g A_f) \to K_m'(A_f)$ is zero. This completes the proof of the subclaim.

We have just shown that if $f,g \in A$ are square-free (or units), then the localization long exact sequence for $A_f/g A_f$, $A_f$ and $A_{fg}$ breaks up into short exact sequences:
\[0 \to K_i'(A_f) \to K'_i(A_{fg}) \to K'_{i-1}(A_f/g A_f) \to 0\]
for all $i \geq 1$. We see at once that $K_i'(A_h)$ injects into $K_i(F)$ for any square-free $h$. On the other hand, if we take $f=1$, $g = X$, and tensor with $\mathbb{Q}$, we see that $K_{i-1}(E)_{\mathbb{Q}} = K'_{i-1}(A/gA)_{\mathbb{Q}}$ non-canonically injects into $K'_i(A_g)_{\mathbb{Q}} \hookrightarrow K_i(F)_{\mathbb{Q}}$. If $i \leq n$, then by induction, $K_{i-1}(E)_{\mathbb{Q}} \neq 0$ since $\trdeg(E/\mathbb{F}_p) \geq n-1$ by construction.
\end{proof}

While the Beilinson-Parshin Conjecture remains open in general, Harder \cite{Harder} has proved it for curves; and from this result we can extract the following:

\begin{cor}\label{torsion}Let $E$ be a finitely-generated function field over $\mathbb{F}_p$ such that $\trdeg(E/\mathbb{F}_p) = 1$. Let $L$ be a purely transcendental extension of $E$ of degree one. Then:
\begin{enumerate}
\item[(i)] $K_n(E)_{\mathbb{Q}} = 0$ for $n > 1$.
\item[(ii)] $K_n(L)_{\mathbb{Q}} = 0$ for $n > 2$.
\end{enumerate}
\end{cor}
\begin{proof}For (i), let $X$ be a smooth, projective curve over $\mathbb{F}_p$ with function field $E$. For any $x \in X$ which is not the generic point, $k(x)$ is just a finite field. As $K_n(\mathbb{F}_q)$ is finite for $n > 1$ \cite[Thm. 8]{Quillen2}, the conditions of Lemma \ref{parshin_special} are satisfied, thereby giving isomorphisms $K_n(X)_{\mathbb{Q}} \cong K_n(E)_{\mathbb{Q}}$ for all $n > 1$. On the other hand, Harder's result \cite{Harder} shows that $K_n(X)_{\mathbb{Q}} = 0$ for $n \geq 1$.

For (ii), we note that $\mathbb{P}^1 \times X$ is a smooth, projective model for $L$ when $X$ is the curve from part (i). If $x \in \mathbb{P}^1 \times X$ is not the generic point of $\mathbb{P}^1 \times X$, it is either closed or the generic point for a curve. By part (i), we therefore have $K_n(k(x))_{\mathbb{Q}} = 0$ for $n > \trdeg(k(x)/\mathbb{F}_p)$. Once again, Lemma \ref{parshin_special} applies, so $K_n(L)_{\mathbb{Q}} \cong K_n(\mathbb{P}^1 \times X)_{\mathbb{Q}}$ for $n > 2$. On the other hand, \cite[VIII.2.1]{Quillen} tells us that $K_n(\mathbb{P}^1 \times X) \cong K_n(X)^{\oplus 2}$, and the latter group is torsion for $n > 0$.
\end{proof}

\begin{prop}Suppose $(R,\pi R, k)$ is a DVR whose residue field $k$ is as in Corollary \ref{torsion}. Then $R$ satisfies Gersten's Conjecture.
\end{prop}
\begin{proof}We need to show that the transfer map $t:K_n(k) \to K_n(R)$ is zero for all $n \geq 0$. For $n \geq 3$, the result follows from Lemmas \ref{dvrintegral} and \ref{torsion}. For $n = 0,1$, the argument is elementary and follows from the localization long exact sequence; for $n=2$, it follows from the fact that every element of $K_2(k)$ is decomposable (Matsumoto's Theorem) --- see \cite[A.2]{Bloch} or \cite[V.6.6.2]{Kbook}.
\end{proof}

\section{Bloch's Formula and the Claborn-Fossum Conjecture}\label{chow}When $X$ is an equicharacteristic regular scheme, Lemmas \ref{panin_sheaf} and \ref{gersten_chow} allow the Chow group on codimension-$p$ cycles to be realized via what is known as Bloch's Formula:
\[ H^p_{\Zar}(X,\mck_p) \cong \ch^p(X). \]
This isomorphism is heavily dependent on the fact that every local ring $\mco_{X,x}$ of $X$ satisfies Gersten's Conjecture. In our mixed characteristic setting, we only have partial acyclicity of the Gersten complex; and in what follows, we shall develop analogues of Bloch's Formula for these more general contexts. See Section \ref{codimension} for a precise definition of Chow group at this level of generality and remarks regarding catenary concerns.

When $X$ is geometrically regular over a one-dimensional base, we have the following identification which may be regarded as a generalization of \cite[Cor. 8]{GilletLevine}:

\begin{cor}\label{chow_zar}Let $Y$ be an integral, regular Noetherian scheme of dimension $1$ with generic point $\eta$. Suppose $X \to Y$ is geometrically regular. Then for all $p \geq 1$, there are isomorphisms:
\[ \ch^p(X) \cong \mathbb{H}^p_{\Zar}\left(X, \left( 0 \to (u_{\eta})_*(\mck_{p,X_{\eta}}) \stackrel{\partial}{\longrightarrow} \bigoplus_{y \in Y^1}(u_y)_*(\mck_{p-1, X_y}) \to 0 \right)\right). \]
(Here, $u_t$ is the natural inclusion of the fibre $X_t \to X$.)
\end{cor}
\begin{proof}For any bounded complex $\mathcal{C}^{\cdot}$ of Zariski sheaves on $X$, there is a hypercohomology spectral sequence
\[ E_2^{pq} = H^p_{\Zar}(X,H^q(\mathcal{C}^{\cdot})) \Rightarrow \mathbb{H}^{p+q}_{\Zar}(X,\mathcal{C}^{\cdot}) \]
The quasi-isomorphism $\left( 0 \to (u_{\eta})_*(\mck_{p,X_{\eta}}) \stackrel{\partial}{\longrightarrow} \bigoplus_{y \in Y^1}(u_y)_*(\mck_{p-1, X_y}) \to 0 \right) \stackrel{\sim}{\longrightarrow} \dumcg_{p,X}$ from Lemma \ref{unramified_bloch} induces the desired isomorphism of hypercohomology; and by Lemmas \ref{hypercohomology} and \ref{gersten_chow}, we have
\[\mathbb{H}^p_{\Zar}\left(\dumcg_{p,X}\right) \cong H^p(\mcg_p^{\cdot}(X)) \cong \ch^p(X). \]
\end{proof}

Since $K$-theory has an inherent product structure, one would hope that such an isomorphism would endow $\ch^*(X)$ with an integral, multiplicative structure --- see \cite{Grayson} for the equicharacteristic case. Note that since $X \to Y$ is not necessarily of finite-type, the usual geometric recipe of reduction to the diagonal technique \cite[8.1]{Fulton} does not apply as $X \times_Y X$ may not even be Noetherian!

When $X$ is geometrically regular over a base of dimension two or higher, we are forced to use the Nisnevich topology since our underlying Theorem \ref{mainthm} only works for Henselian rings.

\begin{cor}\label{chow_nis}Let $X \to Y$ be a geometrically regular morphism of Noetherian schemes, and let $d = \dim Y < \infty$. Then for each $p \geq 0$, there are isomorphisms
\[ \ch^p(X) = \mathbb{H}_{\Nis}^{p}\left( X, \tau^{\leq d}\left(\dumcg_{p,X}\right)\right). \]
\end{cor} 
\begin{proof}The proof is entirely analogous to Corollary \ref{chow_zar}: From Corollary \ref{nisnevich_acyclicity}, 
\[\tau^{\leq d}\left(\dumcg_{n,X}\right) \to \dumcg_{n,X}\]
 is a quasi-isomorphism on $X_{\Nis}$, so the two complexes have the same Nisnevich hypercohomology. The statement now follows from Lemmas \ref{hypercohomology} and \ref{gersten_chow}.
\end{proof}

\subsection{Chow Groups of Local Rings} The classical theorem of Auslander and Buchsbaum \cite{AuslanderBuchsbaum} asserts that a regular local ring $A$ is a unique factorization domain, and hence the divisor-class group $\ch^1(A) = 0$. The following conjecture, proposed by Claborn and Fossum, asserts that same should hold for algebraic cycles of any (positive) codimension:

\begin{conj}[Claborn-Fossum Conjecture]\label{CF}\cite{ClabornFossum} If $(A,\fm_A,L)$ is a regular local ring, then $\ch^p(A) = 0$ for all $p > 0$.
\end{conj}

Since we can identify $\ch^p(A)$ with the $p$-th cohomology group of the Gersten complex on $A$ (Lemma \ref{gersten_chow}), it is clear that this triviality of local Chow groups is implied by Gersten's Conjecture. Thus, The Claborn-Fossum Conjecture \ref{CF} has been settled in the equicharacteristic case by the work of Panin \cite{Panin} and Quillen \cite[VII.5.11]{Quillen}; when $A$ is essentially smooth over a DVR, the result follows from \cite{GilletLevine}. For a non-$K$-theoretic proof, see \cite{DuttaChowII} and \cite{DuttaSmooth}.

\begin{cor}\label{cf_results}Let $(R,\fm_R,k) \to (A,\fm_A,L)$ be a geometrically regular map of Noetherian local rings.
\begin{enumerate}
\item[(i)] If $A$ is Henselian, $\ch^n(A) = 0$ for all $n > \dim R$.
\item[(ii)] If $R$ is a discrete valuation ring, then $\ch^n(A) = 0$ for all $n > 0$.
\end{enumerate}
\end{cor}
\begin{proof}(i) is an immediate consequence of Theorem \ref{mainthm}: $\ch^n(A) = H^n(\mcg_n^{\cdot}(A)) = 0$ for $n > \dim R$.

For (ii), we note that by Lemma \ref{gl_improved}, there are, for each $n \geq 0$, quasi-isomorphisms of complexes
\[ \textstyle \left( 0 \to  K_n(A[\ovpi]) \stackrel{\partial}{\longrightarrow} K_{n-1}(A/\pi A) \to 0 \right) \stackrel{\sim}{\longrightarrow} \mcg_n^{\cdot}(A). \]
Hence, $\ch^n(A) = H^n(\mcg_n^{\cdot}(A)) = 0$ for $n > 1$. On the other hand, $\ch^1(A) = 0$ since $A$ is a UFD.
\end{proof}

\begin{exmp}Be advised that the Corollary \ref{cf_results}(i) says nothing about $\ch^i(A)$ when $i \leq \dim R$. For a classical example, let $R$ be the localization of $\mathbb{C}[X,Y,Z]/(X^2 - Y^5 - Z^7)$ at the origin. By \cite[VII.\S 3, Exercise 7]{Bourbaki}, $R$ is a two-dimensional UFD (in fact, the affine ring $\mathbb{C}[X,Y,Z]/(X^2 - Y^5 - Z^7)$ is as well). If we take $A$ to be the Henselization or completion of $R$, then $R \to A$ is geometrically regular (Example \ref{geomregexmps}) with $\ch^1(R) = 0$. However, by \cite[25.1]{Lipman}, $\ch^1(A) \neq 0$ as $A$ is not UFD.
\end{exmp}


\section*{Acknowledgments}
Over the course of the project, the author benefited from helpful discussions with B. Antieau, H. Gillet, M. Walker, and M. Woolf. This work is partially supported by an NSF RTG Grant (DMS-1246844).

\bibliographystyle{alpha}
\bibliography{relative_gersten}

\begin{thebibliography}{MVW06}

\bibitem[AB59]{AuslanderBuchsbaum}
Maurice Auslander and D.~A. Buchsbaum.
\newblock Unique factorization in regular local rings.
\newblock {\em Proc. Nat. Acad. Sci. U.S.A.}, 45:733--734, 1959.

\bibitem[And74]{Andre}
Michel Andr\'e.
\newblock Localisation de la lissit\'e formelle.
\newblock {\em Manuscripta Math.}, 13:297--307, 1974.

\bibitem[Blo86]{Bloch}
S.~Bloch.
\newblock A note on {G}ersten's conjecture in the mixed characteristic case.
\newblock In {\em Applications of algebraic {$K$}-theory to algebraic geometry
  and number theory, {P}art {I}, {II} ({B}oulder, {C}olo., 1983)}, volume~55 of
  {\em Contemp. Math.}, pages 75--78. Amer. Math. Soc., Providence, RI, 1986.

\bibitem[Bou72]{Bourbaki}
Nicolas Bourbaki.
\newblock {\em Elements of mathematics. {C}ommutative algebra}.
\newblock Hermann, Paris; Addison-Wesley Publishing Co., Reading, Mass., 1972.
\newblock Translated from the French.

\bibitem[CF68]{ClabornFossum}
Luther Claborn and Robert Fossum.
\newblock Generalizations of the notion of class group.
\newblock {\em Illinois J. Math.}, 12:228--253, 1968.

\bibitem[Dut95]{DuttaChowII}
S.~P. Dutta.
\newblock On {C}how groups and intersection multiplicity of modules. {II}.
\newblock {\em J. Algebra}, 171(2):370--382, 1995.

\bibitem[Dut00]{DuttaSmooth}
S.~P. Dutta.
\newblock A theorem on smoothness---{B}ass-{Q}uillen, {C}how groups and
  intersection multiplicity of {S}erre.
\newblock {\em Trans. Amer. Math. Soc.}, 352(4):1635--1645, 2000.

\bibitem[Ful98]{Fulton}
William Fulton.
\newblock {\em Intersection theory}, volume~2 of {\em Ergebnisse der Mathematik
  und ihrer Grenzgebiete. 3. Folge. A Series of Modern Surveys in Mathematics
  [Results in Mathematics and Related Areas. 3rd Series. A Series of Modern
  Surveys in Mathematics]}.
\newblock Springer-Verlag, Berlin, second edition, 1998.

\bibitem[Gei98]{Geisser}
Thomas Geisser.
\newblock Tate's conjecture, algebraic cycles and rational {$K$}-theory in
  characteristic {$p$}.
\newblock {\em $K$-Theory}, 13(2):109--122, 1998.

\bibitem[Gil86]{GilletDVR}
Henri Gillet.
\newblock Gersten's conjecture for the {$K$}-theory with torsion coefficients
  of a discrete valuation ring.
\newblock {\em J. Algebra}, 103(1):377--380, 1986.

\bibitem[GL87]{GilletLevine}
Henri Gillet and Marc Levine.
\newblock The relative form of {G}ersten's conjecture over a discrete valuation
  ring: the smooth case.
\newblock {\em J. Pure Appl. Algebra}, 46(1):59--71, 1987.

\bibitem[GL00]{GeisserLevine}
Thomas Geisser and Marc Levine.
\newblock The {$K$}-theory of fields in characteristic {$p$}.
\newblock {\em Invent. Math.}, 139(3):459--493, 2000.

\bibitem[Gra78]{Grayson}
Daniel~R. Grayson.
\newblock Products in {$K$}-theory and intersecting algebraic cycles.
\newblock {\em Invent. Math.}, 47(1):71--83, 1978.

\bibitem[Gro61]{EGAIII}
A.~Grothendieck.
\newblock \'{E}l\'ements de g\'eom\'etrie alg\'ebrique. {III}. \'{E}tude
  cohomologique des faisceaux coh\'erents. {I}.
\newblock {\em Inst. Hautes \'Etudes Sci. Publ. Math.}, (11):167, 1961.

\bibitem[GS87]{GilletSoule}
H.~Gillet and C.~Soul\'e.
\newblock Intersection theory using {A}dams operations.
\newblock {\em Invent. Math.}, 90(2):243--277, 1987.

\bibitem[Har77]{Harder}
G.~Harder.
\newblock Die {K}ohomologie {$S$}-arithmetischer {G}ruppen \"uber
  {F}unktionenk\"orpern.
\newblock {\em Invent. Math.}, 42:135--175, 1977.

\bibitem[Kah05]{Kahn}
Bruno Kahn.
\newblock Algebraic {$K$}-theory, algebraic cycles and arithmetic geometry.
\newblock In {\em Handbook of {$K$}-theory. {V}ol. 1, 2}, pages 351--428.
  Springer, Berlin, 2005.

\bibitem[Lip69]{Lipman}
Joseph Lipman.
\newblock Rational singularities, with applications to algebraic surfaces and
  unique factorization.
\newblock {\em Inst. Hautes \'Etudes Sci. Publ. Math.}, (36):195--279, 1969.

\bibitem[Mat86]{Matsumura}
Hideyuki Matsumura.
\newblock {\em Commutative ring theory}, volume~8 of {\em Cambridge Studies in
  Advanced Mathematics}.
\newblock Cambridge University Press, Cambridge, 1986.
\newblock Translated from the Japanese by M. Reid.

\bibitem[Mil80]{Milne}
James~S. Milne.
\newblock {\em \'Etale cohomology}, volume~33 of {\em Princeton Mathematical
  Series}.
\newblock Princeton University Press, Princeton, N.J., 1980.

\bibitem[MR10]{Majadas}
Javier Majadas and Antonio~G. Rodicio.
\newblock {\em Smoothness, regularity and complete intersection}, volume 373 of
  {\em London Mathematical Society Lecture Note Series}.
\newblock Cambridge University Press, Cambridge, 2010.

\bibitem[MVW06]{Mazza}
Carlo Mazza, Vladimir Voevodsky, and Charles Weibel.
\newblock {\em Lecture notes on motivic cohomology}, volume~2 of {\em Clay
  Mathematics Monographs}.
\newblock American Mathematical Society, Providence, RI; Clay Mathematics
  Institute, Cambridge, MA, 2006.

\bibitem[Nis89]{Nisnevich}
Ye.~A. Nisnevich.
\newblock The completely decomposed topology on schemes and associated descent
  spectral sequences in algebraic {$K$}-theory.
\newblock In {\em Algebraic {$K$}-theory: connections with geometry and
  topology ({L}ake {L}ouise, {AB}, 1987)}, volume 279 of {\em NATO Adv. Sci.
  Inst. Ser. C Math. Phys. Sci.}, pages 241--342. Kluwer Acad. Publ.,
  Dordrecht, 1989.

\bibitem[Pan03]{Panin}
I.~A. Panin.
\newblock The equicharacteristic case of the {G}ersten conjecture.
\newblock {\em Tr. Mat. Inst. Steklova}, 241(Teor. Chisel, Algebra i Algebr.
  Geom.):169--178, 2003.

\bibitem[Pes66]{Peskine}
Christian Peskine.
\newblock Une g\'en\'eralisation du ``main theorem'' de {Z}ariski.
\newblock {\em Bull. Sci. Math. (2)}, 90:119--127, 1966.

\bibitem[Pop85]{Popescu}
Dorin Popescu.
\newblock General {N}\'eron desingularization.
\newblock {\em Nagoya Math. J.}, 100:97--126, 1985.

\bibitem[Qui72]{Quillen2}
Daniel Quillen.
\newblock On the cohomology and {$K$}-theory of the general linear groups over
  a finite field.
\newblock {\em Ann. of Math. (2)}, 96:552--586, 1972.

\bibitem[Qui73]{Quillen}
Daniel Quillen.
\newblock Higher algebraic {$K$}-theory. {I}.
\newblock In {\em Algebraic {$K$}-theory, {I}: {H}igher {$K$}-theories ({P}roc.
  {C}onf., {B}attelle {M}emorial {I}nst., {S}eattle, {W}ash., 1972)}, pages
  85--147. Lecture Notes in Math., Vol. 341. Springer, Berlin, 1973.

\bibitem[RS90]{ReidSherman}
L.~Reid and C.~Sherman.
\newblock The relative form of {G}ersten's conjecture for power series over a
  complete discrete valuation ring.
\newblock {\em Proc. Amer. Math. Soc.}, 109(3):611--613, 1990.

\bibitem[SGA72]{SGA4.2}
{\em Th\'eorie des topos et cohomologie \'etale des sch\'emas. {T}ome 2}.
\newblock Lecture Notes in Mathematics, Vol. 270. Springer-Verlag, Berlin-New
  York, 1972.
\newblock S\'eminaire de G\'eom\'etrie Alg\'ebrique du Bois-Marie 1963--1964
  (SGA 4), Dirig\'e par M. Artin, A. Grothendieck et J. L. Verdier. Avec la
  collaboration de N. Bourbaki, P. Deligne et B. Saint-Donat.

\bibitem[Ska16]{Koszul}
C.~Skalit.
\newblock Koszul factorization and the {C}ohen-{G}abber theorem.
\newblock {\em Illinois J. Math.}, 60(3-4):833--844, 2016.

\bibitem[Sri91]{Srinivas}
V.~Srinivas.
\newblock {\em Algebraic {$K$}-theory}, volume~90 of {\em Progress in
  Mathematics}.
\newblock Birkh\"auser Boston, Inc., Boston, MA, 1991.

\bibitem[Swa98]{Swan}
Richard~G. Swan.
\newblock N\'eron-{P}opescu desingularization.
\newblock In {\em Algebra and geometry ({T}aipei, 1995)}, volume~2 of {\em
  Lect. Algebra Geom.}, pages 135--192. Int. Press, Cambridge, MA, 1998.

\bibitem[Tam94]{Tamme}
G\"unter Tamme.
\newblock {\em Introduction to \'etale cohomology}.
\newblock Universitext. Springer-Verlag, Berlin, 1994.
\newblock Translated from the German by Manfred Kolster.

\bibitem[Wei94]{Weibel}
Charles~A. Weibel.
\newblock {\em An introduction to homological algebra}, volume~38 of {\em
  Cambridge Studies in Advanced Mathematics}.
\newblock Cambridge University Press, Cambridge, 1994.

\bibitem[Wei13]{Kbook}
Charles~A. Weibel.
\newblock {\em The {$K$}-book}, volume 145 of {\em Graduate Studies in
  Mathematics}.
\newblock American Mathematical Society, Providence, RI, 2013.
\newblock An introduction to algebraic $K$-theory.

\end{thebibliography}
\end{document}